\documentclass[a4paper,12pt, twoside, pdftex]{amsart}

\usepackage{fullpage}
\usepackage{amsmath, amsthm, amssymb}
\usepackage{txfonts}
\usepackage{amsfonts}
\usepackage{prettyref}
\usepackage{mathrsfs}
\usepackage[all]{xy}
\usepackage[utf8]{inputenc}

\newtheorem{dummy}{dummy}[section]
\newtheorem{lemma}[dummy]{Lemma}
\newtheorem{theorem}[dummy]{Theorem}

\newtheorem{corollary}[dummy]{Corollary}

\theoremstyle{definition}
\newtheorem{definition}[dummy]{Definition}
\newtheorem*{definition*}{Definition}
\newtheorem{example}[dummy]{Example}

\newtheorem{remark}[dummy]{Remark}

\newtheorem*{acknowledgements}{Acknowledgements}

\numberwithin{equation}{section}

\usepackage{tikz}
\usetikzlibrary{backgrounds}

\newrefformat{th}{Theorem~\ref{#1}}
\newrefformat{cr}{Corollary~\ref{#1}}
\newrefformat{lm}{Lemma~\ref{#1}}
\newrefformat{dl}{Definition-Lemma~\ref{#1}}
\newrefformat{df}{Definition~\ref{#1}}
\newrefformat{cl}{Claim~\ref{#1}}
\newrefformat{sl}{Sublemma~\ref{#1}}
\newrefformat{pr}{Proposition~\ref{#1}}
\newrefformat{cj}{Conjecture~\ref{#1}}
\newrefformat{st}{Step~\ref{#1}}
\newrefformat{sc}{Section~\ref{#1}}
\newrefformat{df}{Definition~\ref{#1}}
\newrefformat{rm}{Remark~\ref{#1}}
\newrefformat{q}{Question~\ref{#1}}
\newrefformat{pb}{Problem~\ref{#1}}
\newrefformat{cd}{Condition~\ref{#1}}
\newrefformat{example}{Example~\ref{#1}}
\newrefformat{he}{Heore~\ref{#1}}
\newrefformat{fg}{Figure~\ref{#1}}
\newrefformat{tb}{Table~\ref{#1}}
\newrefformat{as}{Assumption~\ref{#1}}

\newcommand{\pref}{\prettyref}
\newcommand{\can}{\operatorname{can}}
\newcommand{\cont}{\operatorname{cont}}
\newcommand{\Coh}{\operatorname{Coh}}

\newcommand{\Cone}{\operatorname{Cone}}
\newcommand{\Conv}{\operatorname{Conv}}
\newcommand{\Core}{\operatorname{Core}}

\newcommand{\del}{\partial}

\newcommand{\Exit}{\operatorname{Exit}}

\newcommand{\Fan}{\operatorname{Fan}}
\newcommand{\Fuk}{\operatorname{Fuk}}

\newcommand{\LGr}{\operatorname{LGr}}

\newcommand{\Hom}{\operatorname{Hom}}

\newcommand{\Int}{\operatorname{Int}}
\newcommand{\Lie}{\operatorname{Lie}}
\newcommand{\Log}{\operatorname{Log}}

\newcommand{\mom}{\operatorname{mom}}
\newcommand{\msh}{\operatorname{\mu sh}}

\newcommand{\Nbd}{\operatorname{Nbd}}

\newcommand{\pr}{\operatorname{pr}}

\newcommand{\ret}{\operatorname{ret}}

\newcommand{\Sk}{\operatorname{Sk}}

\newcommand{\stFan}{\operatorname{stFan}}

\newcommand{\VVert}{\operatorname{Vert}}

\newcommand{\cL}{\mathcal{L}}
\newcommand{\cM}{\mathcal{M}}
\newcommand{\cN}{\mathcal{N}}

\newcommand{\cT}{\mathcal{T}}

\newcommand{\bA}{\mathbb{A}}
\newcommand{\bC}{\mathbb{C}}

\newcommand{\bL}{\mathbb{L}}

\newcommand{\bP}{\mathbb{P}}

\newcommand{\bR}{\mathbb{R}}
\newcommand{\bT}{\mathbb{T}}
\newcommand{\bZ}{\mathbb{Z}}

\newcommand{\bfT}{\mathbf{T}}

\newcommand{\bfW}{\mathbf{W}}

\begin{document}

\title{Large volume limit fibrations over fanifolds}
\author[H.~Morimura]{Hayato Morimura}
\address{SISSA, via Bonomea 265, 34136 Trieste, Italy}
\email{hayato.morimura@ipmu.jp}
\curraddr{
Kavli Institute for the Physics and Mathematics of the Universe (WPI),
University of Tokyo,
5-1-5 Kashiwanoha,
Kashiwa,
Chiba,
277-8583,
Japan.}

\date{}
\pagestyle{plain}

\begin{abstract}
We lift the stratified torus fibration over a fanifold constructed by Gammage--Shende to the associated Weinstein manifold-with-boundary,
which is homotopic to a filtered stratified integrable system with noncompact fibers.
When the fanifold admits a dual stratified space in a suitable sense,
we give a stratified fibration over it completing SYZ picture.
For the fanifold associated with a very affine hypersurface,
we realize the latter fibration as a restriction of SYZ fibrations over the tropical hypersurface proposed by Abouzaid--Auroux--Katzarkov. 
\end{abstract}

\maketitle

\section{Introduction}
Given a Calabi--Yau manifold
$X$,
Homological Mirror Symmetry(HMS) conjecture
\cite{Kon}
claims the existence of another Calabi--Yau manifold
$\breve{X}$ 
called its mirror partner
whose
Fukaya category
$\Fuk(\breve{X})$
and
dg category of coherent sheaves
$\Coh(\breve{X})$
are respectively equivalent to
$\Coh(X)$
and
$\Fuk(X)$.
The equivalences are believed to connect the
symplectic
and
complex
geometry of
the mirror pair
$(X, \breve{X})$.
Nowadays,
HMS concerns more general spaces
but neither
precise definition
nor
systematic construction
of mirror pairs are available.

One way to go is indicated by Strominger--Yau--Zaslow(SYZ) conjecture
\cite{SYZ},
which has been elaborated independently by
Kontsevich--Soibelman
\cite{KS00, KS06}
and
Gross--Siebert
\cite{GS11}. 
Roughly speaking,
$(X, \breve{X})$
should admit so called SYZ fibrations,
i.e.
dual special Lagrangian torus fibrations
$X \to B \leftarrow \breve{X}$
over a common base.
Then
$\breve{X}$
should be obtained by
dualizing
$X \to B$
over the smooth locus
and
compactifying the result in a suitable way. 
SYZ conjecture is
hard to correctly formulate
and
still largely open.

Recently,
Gammage--Shende introduced fanifolds
\cite[Definition 2.4]{GS1}
by gluing rational polyhedral fans of cones,
which provide the organizing
topological
and
discrete
data for HMS at large volume
\cite[Theorem 5.4]{GS1}. 
To a fanifold
$\Phi$
they associated an algebraic space
$\bfT(\Phi)$
\cite[Proposition 3.10]{GS1}
obtained as the gluing of the toric varieties
$T_\Sigma$
associated with the fans
$\Sigma$
along their toric boundaries.
Here,
one could consider
stacky fans
and
their associated toric stacks
instead. 
Based on an idea from SYZ fibrations,
they constructed its mirror partner
$\widetilde{\bfW}(\Phi)$
by inductive Weinstein handle attachments. 

\begin{theorem}[{\cite[Theorem 4.1]{GS1}}] \label{thm:fibration}
Let
$\Phi \subset \cM$
be a fanifold.
Then there exists a triple
$(\widetilde{\bfW}(\Phi), \widetilde{\bL}(\Phi), \pi)$
of
a subanalytic Weinstein manifold-with-boundary
$\widetilde{\bfW}(\Phi)$,
a conic subanalytic Lagrangian
$\widetilde{\bL}(\Phi) \subset \widetilde{\bfW}(\Phi)$,
and
a map
$\pi \colon \widetilde{\bL}(\Phi) \to \Phi$
satisfying the following conditions.
\begin{itemize}
\item[(1)]
Let
$S \subset \Phi$
be a stratum of codimension
$d$.
Then:
\begin{itemize}
\item[$\bullet$]
$\pi^{-1}(S) \cong T^d \times S$
where
$T^d$
is a real $d$-dimensional torus.
\item[$\bullet$]
$\pi^{-1}(\Nbd(S)) \cong \bL(\Sigma_S) \times S$
where
$\Nbd(S)$
is a suitable neighborhood
and
$\bL(\Sigma_S)$
is the FLTZ Lagrangian associated with the normal fan
$\Sigma_S$
of
$S$.
\item[$\bullet$]
In a neighborhood of
$\pi^{-1}(S) \cong T^d \times S$,
there is a symplectomorphism of pairs
\begin{align} \label{eq:charts}
(T^* T^d \times T^* S, \bL(\Sigma_S) \times S)
\hookrightarrow
(\widetilde{\bfW}(\Phi), \widetilde{\bL}(\Phi)).
\end{align}
\end{itemize}
\item[(2)]
If
$\Phi$
is closed,
then we have
$\widetilde{\bL}(\Phi) = \Core(\widetilde{\bfW}(\Phi))$
for the skeleton
$\Core(\widetilde{\bfW}(\Phi))$
of
$\widetilde{\bfW}(\Phi)$.
\item[(3)]
A subfanifold
$\Phi^\prime \subset \Phi$
determines a stopped Weinstein sector
$\widetilde{\bfW}(\Phi^\prime)
\subset
\widetilde{\bfW}(\Phi)$
with relative skeleton
$\widetilde{\bL}(\Phi^\prime)
=
\widetilde{\bfW}(\Phi^\prime) \cap \widetilde{\bL}(\Phi)$
\item[(4)]
The Weinstein manifold-with-boundary
$\widetilde{\bfW}(\Phi)$
carries a Lagrangian polarization given in the local charts
\pref{eq:charts}
by taking 
the fiber direction in
$T^* T^d$
and
the base direction in
$T^* S$.
\end{itemize}
\end{theorem}

The guiding principle behind their construction of
$\widetilde{\bfW}(\Phi)$
was that
the FLTZ Lagrangians
$\bL(\Sigma)$
and
the projections
$\bL(\Sigma) \to \Sigma$
should glue to yield
$\Core(\widetilde{\bfW}(\Phi))$
and
$\pi$.
\pref{thm:fibration}(4)
gives rise to canonical grading/orientation data
\cite[Section 5.3]{GPS3},
which one needs to define the partially wrapped Fukaya category
$\Fuk(\widetilde{\bfW}(\Phi), \del_\infty \widetilde{\bL}(\Phi))$.
By
\cite[Theorem 1.4]{GPS3}
we have
\begin{align*}
\Fuk(\widetilde{\bfW}(\Phi), \del_\infty \widetilde{\bL}(\Phi))
\cong
\Gamma (\widetilde{\bL}(\Phi), \msh_{\widetilde{\bL}(\Phi)})^{op}
\cong
\Gamma (\Phi, \pi_* \msh_{\widetilde{\bL}(\Phi)})^{op}
\end{align*}
where
$\msh_{\widetilde{\bL}(\Phi)}$
is a certain constructible sheaf of dg categories.
Over neighborhoods of strata,
local sections of
$\pi_* \msh_{\widetilde{\bL}(\Phi)}$
are computed in
\cite[Section 7.3]{GS2}.
The third bullet in
\pref{thm:fibration}(1)
gives rise to gluing data
\cite[Proposition 4.34]{GS1},
which one needs to determine
$\pi_* \msh_{\widetilde{\bL}(\Phi)}$.

SYZ fibrations often give candidates for mirror pairs
and
\cite[Theorem 5.4]{GS1}
is just one of such examples.
Moreover,
\cite[Corollary 5.8]{GS1}
implies that
HMS at large volume limit is obtained by the local-to-global principle.
Hence it is natural to expect that
the canonical lift of the projections
$\bL(\Sigma) \to \Sigma$
should glue to yield a version of $A$-side SYZ fibration
$\widetilde{\bfW}(\Phi) \to \Phi$,
as predicted by Gammage--Shende
\cite[Remark 4.5]{GS1}.
Indeed,
taking wrapped Fukaya categories
and
gluing of local pieces of 
$\widetilde{\bfW}(\Phi)$,
which respects the gluing of
$\bL(\Sigma)$,
intertwine in the sense of
\cite[Corollary 5.9]{GS1}.
In this paper,
we prove the following.

\begin{theorem} \label{thm:lift}
There is a filtered stratified fibration
$\bar{\pi} \colon \widetilde{\bfW}(\Phi) \to \Phi$
restricting to
$\pi$,
which is homotopic to a filtered stratified integrable system with noncompact fibers.
If the normal fan
$\Sigma_S$
associated to all stratum
$S \subset \Phi$
is proper,
then the homotopy becomes trivial. 
\end{theorem}

Note that
by construction of
$\widetilde{\bfW}(\Phi)$
one can expect isotrivial fibers only on each stratum of the induced stratification by
$\Phi$.
Moreover,
$\Phi$
admits a filtration
\cite[Remark 2.12]{GS1}
and
fibers would vary
when passing through the induced filters.
Since in general the canonical lift of
$\pi$
does not land in
$\Phi$,
we must compose the map induced by a retraction,
which might not be smooth,
even
$C^1$.
After that,
as suggested in
\cite[Remark 4.5]{GS1},
it remains to show topological compatibility with the handle attachment process.
Then we need to fully understand the proofs of
\pref{thm:fibration},
especially
(4)
and
the third bullet in
(1).
As their details are skipped in
\cite{GS1},
we include them for completeness.

The map
$\bar{\pi}$
is not parallel to the original $A$-side SYZ fibration which should connect
$\widetilde{\bfW}(\Phi)$
to
$\bfT(\Phi)$,
as the expected base is dual stratified space to
$\Phi$
in a certain sense.
Working over
$\Phi$
amounts to removing the requirement for the moment polytopes of
$T_\Sigma$
to glue to yield the expected base.
Indeed,
Gammage--Shende introduced fanifolds to realize more general mirror partners via such an extended version of SYZ duality.
Hence,
as mentioned in
\cite[Remark 4.5]{GS1},
if
$\Phi$
admits a well defined dual stratified space
$\Psi$,
then one expects the gluing
$\bfT(\Phi) \to \Psi$
of moment maps
to be the $B$-side SYZ fibration.
We obtain a stratified fibration
which should be its SYZ dual,
modifying the canonical lift of
$\pi$.

\begin{theorem} \label{thm:dual}
Assume that
$\Phi$
admits a dual stratified space
$\Psi$
in the sense of Definition
\pref{dfn:dual}.
Then there is a stratified fibration
$\underline{\pi} \colon \widetilde{\bfW}(\Phi) \to \Psi$.
Let
$S \subset \Phi$
be a stratum of codimension
$d$.
Over a point of its dual stratum
$S^\perp \subset \Psi$,
the fiber of
$\underline{\pi}$
is isomorphic to
$T^d \times T^* S$.
\end{theorem}

Finally,
we provide an evidence for
$\underline{\pi}$
to be the $A$-side SYZ fibration for
$(\widetilde{\bfW}(\Phi), \bfT(\Phi))$.
Consider the fanifold
$\Phi = \Sigma \cap S^{n+1}$
from
\cite[Example 4.22]{GS1}
generalized as in
\cite[Section 6]{GS1}.
It is associated with the mirror pair
$(H, \del \bfT_\Sigma)$
\cite[Theorem 1.0.1]{GS2}
in the sense that
$\Core(\widetilde{\bfW}(\Phi))
\cong
\Core(H)$
and
$\bfT(\Phi) = \del \bfT_\Sigma$
for
a very affine hypersuface
$H \subset (\bC^*)^{n+1}$
and
the toric boundary divisor
$\del \bfT_\Sigma$
of the toric stack
$\bfT_\Sigma$.
The fanifold
$\Phi$
admits a dual stratified space
$\Psi$.
We show that
the SYZ dual fibrations for
$(H, \del \bfT_\Sigma)$
over the tropical hypersurface
$\Pi_\Sigma$
of
$H$
\cite[Section 3]{AAK}
restricts to that for
$(\widetilde{\bfW}(\Phi), \bfT(\Phi))$
over
$\Psi$
in the following sense.

\begin{theorem} \label{thm:restriction}
Let
$\Phi$
be the fanifold from
\cite[Example 4.22]{GS1}
generalized as in
\cite[Section 6]{GS1}.
Then there are
a stratified homeomorphism
$\Psi \hookrightarrow \Pi_\Sigma$,
a symplectomorphism of pairs
$(\widetilde{\bfW}(\Phi), \widetilde{\bL}(\Phi))
\hookrightarrow
(H, \Core(H))$
and
a map
$\bfT(\Phi) \to \Psi$
which makes the diagram
\begin{align*}
\begin{gathered}
\xymatrix{
\widetilde{\bL}(\Phi) \ar@{^{(}->}[r] \ar@{=}[d] & \widetilde{\bfW}(\Phi) \ar^{\underline{\pi}}[r] \ar@{^{(}->}[d] & \Psi \ar@{^{(}->}[d] & \bfT(\Phi) \ar[l] \ar@{^{(}->}[d] \\
\Core(H) \ar@{^{(}->}[r] & H \ar[r]_{} & \Pi_\Sigma & -K_{\bfT_\Sigma} \ar[l],
}
\end{gathered}
\end{align*}
commute,
where
$H \to \Pi_\Sigma \leftarrow -K_{\bfT_\Sigma}$
are canonical extensions to toric stacks of SYZ fibrations from
\cite[Section 3]{AAK}.
\end{theorem}

It turns out that
\pref{thm:lift}
has an interesting application.
Namely,
$\bar{\pi}$
provides a key ingredient to the construction of a Weinstein sectorial cover of
$\widetilde{\bfW}(\Phi)$,
which is compatible with the corresponding closed cover of
$\bfT(\Phi)$
\cite[Theorem 1.2]{Mor}.
As a consequence,
via sectorial descent we obtain a description of HMS for
$(\widetilde{\bfW}(\Phi), \bfT(\Phi))$
as an isomorphism of cosheaves of categories.
We wonder
if the same idea would work for more general SYZ fibrations,
which might give a path from SYZ to homological mirrors as follows,
provided HMS for local pieces.
First,
given enough good SYZ fibrations,
take dual stratified space to the base.
Rewinding the process to obtain
$\underline{\pi}$
from
$\bar{\pi}$,
one modifies the original fibrations to that over the dual. 
Lift a suitable cover of the dual along the modified $A$-side fibration to obtain a Weinstein sectorial cover.
Then carefully glue HMS for local pieces to obtain a global HMS. 
 
\begin{acknowledgements}
The author was supported by SISSA PhD scholarships in Mathematics.
He is grateful to an anonymous referee for
careful reading
and
helpful comments.
This work was partially supported by
JSPS KAKENHI Grant Number
JP23KJ0341.
\end{acknowledgements}

\section{Fanifolds}
Fanifolds are introduced by Gammage--Shende in
\cite[Section 2]{GS1}
as a formulation of stratified manifolds for
which the geometry normal to each stratum is equipped with the structure of a fan. 
Let
$\Phi$
be a stratified space.
As in
\cite{GS1},
throughout the paper,
we assume
$\Phi$
to satisfy the following conditions.

\begin{itemize}
\item[(i)]
$\Phi$
has finitely many strata.
\item[(ii)]
$\Phi$
is conical in the complement of a compact subset. 
\item[(iii)]
$\Phi$
is given as a germ of a closed subset in an ambient manifold
$\cM$.
\item[(iv)]
The strata of
$\Phi$
are smooth submanifolds of
$\cM$.
\item[(v)]
The strata of
$\Phi$
are contractible.
\end{itemize}

We will express properties of
$\Phi$
in terms of the chosen ambient manifold
$\cM$
as long as they only depend on the germ of
$\Phi$.
Taking the normal cone
$C_S \Phi \subset T_S \cM$
for each stratum
$S \subset \Phi$,
one obtains a stratification on
$C_S \Phi$
induced by the canonical stratification of a sufficiently small tubular neighborhood
$T_S \cM \to \cM$. 
We will call it the
\emph{canonically induced stratification}
by
$\cM$.

\begin{definition}
The stratified space
$\Phi$
is
\emph{smoothly normally conical}
if for each stratum
$S \subset \Phi$
its sufficiently small tubular neighborhood
$T_S \cM \to \cM$
induces locally near
$S$
a stratified diffeomorphism
$C_S \Phi \to \Phi$,
which in turn induces the identity
$C_S \Phi \to C_S \Phi$
where the target is equipped with the canonically induced stratification by
$\cM$.
\end{definition}

\begin{example} \label{eg:SMCSS}
Associating to each cone its interior as a stratum,
one may regard a fan
$\Sigma$
of cones as a stratified space satisfying the conditions
(i), ... , (v).
Clearly,
$\Sigma$
is smoothly normally conical.
By abuse of notation,
we use the same symbol 
$\sigma \in \Sigma$
to denote the stratum of
$\Sigma$
corresponding to
$\sigma$.
We introduce a partial order in
$\Sigma$
canonically descending to the stratified space. 
Namely, 
for
$\sigma, \tau \in \Sigma$
we define
$\sigma < \tau$
if and only if
$\sigma \subset \bar{\tau}$
for the closure
$\bar{\tau}$
of
$\tau$
in
$\Sigma$.
\end{example}

\begin{definition}
We write
$\Exit(\Phi)$
for the exit path category.
For each stratum
$S \subset \Phi$
we write
$\Exit_S(\Phi)$
for the full subcategory of exit paths starting at
$S$
contained inside a sufficiently small neighborhood of
$S$.
\end{definition}

\begin{definition}
We write
$\Fan^\twoheadrightarrow$
for the category whose objects are pairs
$(M, \Sigma)$
of a lattice
$M$
and
a stratified space
$\Sigma$
by finitely many rational polyhedral cones in
$M_\bR = M \otimes_\bZ \bR$.
For any
$(M, \Sigma), (M^\prime, \Sigma^\prime) \in \Fan^\twoheadrightarrow$
a morphism
$(M, \Sigma) \to (M^\prime, \Sigma^\prime)$
is given by
the data of a cone
$\sigma \in \Sigma$
and
an isomorphism
$M / \langle \sigma \rangle \cong M^\prime$
such that
$\Sigma^\prime
=
\Sigma / \sigma
=
\{ \tau / \langle \sigma \rangle \subset M_\bR / \langle \sigma \rangle \  | \ \tau \in \Sigma, \sigma \subset \bar{\tau} \}$.
We denote by
$\Fan^\twoheadrightarrow_{\Sigma /}$
for an object
$(M, \Sigma) \in \Fan^\twoheadrightarrow$
the full subcategory of objects
$(M^\prime, \Sigma^\prime)$
with
$\Sigma^\prime = \Sigma / \sigma$
for some cone
$\sigma \in \Sigma$.
\end{definition}

We have a natural identification
\begin{align*}
\begin{gathered}
\Exit(\Sigma) \cong \Fan^\twoheadrightarrow_{\Sigma /}, \
\sigma \mapsto [\Sigma \mapsto \Sigma / \sigma]
\end{gathered}
\end{align*}
of posets.
In addition,
the normal geometry to
$\sigma$
is the geometry of
$\Sigma / \sigma$.
This is the local model of fanifolds introduced by
Gammage--Shende.

\begin{definition}[{\cite[Definition 2.4]{GS1}}] \label{dfn:fanifold}
A
\emph{fanifold}
is a smoothly normally conical stratified space 
$\Phi \subset \cM$
satisfying the conditions
(i), ..., (v)
and
equipped with the following data:
\begin{itemize}
\item
A functor
$\Exit(\Phi) \to \Fan^\twoheadrightarrow$
whose value on each stratum
$S$
is a pair
$(M_S, \Sigma_S)$
of a lattice
$M_S$
and
a rational polyhedral fan
$\Sigma_S \subset M_{S, \bR}$
called the
\emph{associated normal fan}.
\item
For each stratum
$S \subset \Phi$
a trivialization
$\phi_S \colon T_S \cM \cong M_{S, \bR}$
of the normal bundle carrying the induced stratification on
$C_S \Phi$
to the standard stratification induced by
$\Sigma_S$.
\end{itemize}
These data must make the diagram
\begin{align*}
\begin{gathered}
\xymatrix{
T_S \cM \ar[r]^-{\phi_S} \ar_{}[d] & M_{S, \bR} \ar^{}[d] \\
T_{S^\prime} \cM |_S \ar[r]_-{\phi_{S^\prime}} & M_{S^\prime, \bR},
}
\end{gathered}
\end{align*}
commute for any stratum
$S^\prime$
of the induced stratification on
$\Nbd(S)$,
where the left vertical arrow is the quotient by the span of
$S^\prime$.
The right vertical arrow corresponds to the map
$M_S \to M_{S^\prime}$ 
on lattices.
\end{definition}

\begin{remark}
In the original definition,
the conditions
(i), ..., (iv)
were assumed in advance
and
the condition
(v)
was added later
\cite[Assumption 2.5]{GS1}.
Due to
(v),
$\Exit_S(\Phi)$
is equivalent to the poset
$\Exit(\Sigma_S)$.
Although we have already assumed
$\Phi$
to satisfy
(i), ..., (v)
above,
here we mention it again to emphasize the difference from the original form. 
\end{remark}

\begin{example} \label{eg:trivial}
A contractible manifold
$\cM$
regarded as a trivially stratified space is obviously smoothly normally conical satisfying (i), ..., (v).
Associating to the unique stratum
$\cM$
a pair
$(M_\cM = \{ 0 \}, \Sigma_\cM = \{ 0 \})$
defines a trivial fanifold structure on
$\cM$.
It follows that
the product of
a contractible manifold
and
a fanifold
is canonically a fanifold.
\end{example}


\begin{example}[{\cite[Example 2.7]{GS1}}] \label{eg:fan}
As explained in Example
\pref{eg:SMCSS},
a fan
$\Sigma \subset \bR^n$
of cones regarded as a stratified space
is smoothly normally conical
and
satisfies
(i), ..., (v).
Associating to each stratum
$\sigma \in \Sigma$
a pair
$(M / \langle \sigma \rangle, \Sigma / \sigma)$
of the quotient lattice
$M / \langle \sigma \rangle$
and
the normal fan
$\Sigma / \sigma$
defines a fanifold structure on
$\Sigma$.
\end{example}

\begin{example}[{\cite[Example 2.10]{GS1}}] \label{eg:hypersurface}
Given a fanifold
$\Phi \subset \cM$,
if a submanifold
$\cM^\prime \subset \cM$
intersects transversely all strata of
$\Phi$,
then
$\Phi \cap \cM^\prime \subset \cM^\prime$
canonically inherits the fanifold structure.
In particular,
the ideal boundary
$\del_\infty \Phi$
of the fanifold carries a canonical fanifold structure. 
\end{example}

\begin{lemma}[{\cite[Remark 2.12]{GS1}}] \label{lem:filtration}
Let
$\Phi$
be a fanifold of dimension
$n$.
Then it admits a filtration
\begin{align} \label{eq:filtration1}
\Phi_0 \subset \Phi_1 \subset \cdots \subset \Phi_n = \Phi
\end{align}
where
$\Phi_k$
are subfanifolds defined as sufficiently small neighborhoods of $k$-skeleta
$\Sk_k(\Phi)$,
the closure of the subset of $k$-strata.
\end{lemma}
\begin{proof}
The normal geometry to a $0$-stratum
$P \subset \Phi$
is the geometry of the normal fan
$\Sigma_P \subset M_{P, \bR}$
by definition.
Let
\begin{align*}
\Phi_0 = \bigsqcup_P \Sigma_P
\end{align*}
be the disjoint union of
$\Sigma_P$
equipped with the canonically induced fanifold structure,
where
$P$
runs through all $0$-strata
of
$\Phi$.
Then
$\Phi_0 \subset \Phi$
is clearly a subfanifold containing
$\Sk_0(\Phi)$.


Suppose that
$\Phi_{k-1}$
admits the desired filtration
$\Phi_0 \subset \Phi_1 \subset \cdots \subset \Phi_{k-1}$.
The normal geometry to a $k$-stratum
$S \subset \Phi$
is the geometry of the normal fan
$\Sigma_S \subset M_{S, \bR}$
by definition.
The ideal boundary
$\del_\infty S$
might have some subset
$\del_{in} S$
which is in the direction to the interior of
$\Phi$.
Perform the gluing
$\Phi_{k-1} \#_{(\Sigma_S \times \del_{in} S)} (\Sigma_S \times S)$
which equals
$\Sigma_S \times S$
when
$\Phi_{k-1}$
is empty
and
equals
$\Phi_{k-1}$
unless
$S$
is an
\emph{interior}
$k$-stratum,
i.e.
$\del_{in} S = \del S_\circ$
where
$S_\circ = S \setminus \Phi_{k-1}$
is a manifold-with-boundary.
Note that
by Example
\pref{eg:hypersurface}
the induced fanifold structure on
$\del_{in} S$
is compatible with that of
$\Phi_{k-1}$.
Let
\begin{align*}
\Phi_k = \Phi_{k-1} \#_{\bigsqcup_S \Sigma_S \times \del_{in} S} \bigsqcup_S (\Sigma_S \times S)
\end{align*}
be the result of such gluings for all $k$-strata
$S$
in
$\Phi$.
Then
$\Phi_k \subset \Phi$
is a subfanifold containing
$\Sk_k(\Phi)$
since the products
$\Sigma_S \times S$
are canonically fanifolds
by Example
\pref{eg:trivial},
\pref{eg:fan}.
\end{proof}

\begin{definition}
A fanifold
$\Phi \subset \cM$
is
\emph{closed}
if its all strata are interior.
Namely,
any $k$-stratum
$S$
of
$\Phi$
satisfies
$\del_{in} S = \del S_\circ$
where
$S_\circ = S \setminus \Phi_{k-1}$. 
\end{definition}

\begin{example}[{\cite[Example 4.21]{GS1}}]  \label{eg:closed}
Consider a $2$-dimensional fanifold
$\Phi = [0,1] \times [0,1]$
stratified by
vertices
$P_1 = (0,1),
P_2 = (0,0),
P_3 = (1,0),
P_4 = (1,1)$,
edges
$I_{12} = \{ 0 \} \times (0,1),
I_{23} = (0,1) \times \{ 0 \},
I_{34} = \{ 1 \} \times (0,1),
I_{14} = (0,1) \times \{ 1 \}$
and
a face
$F = (0,1) \times (0,1)$.
All strata are interior
and
$\Phi$
is closed.
The filtration from
\pref{lem:filtration}
is given by
\begin{align} \label{eq:filters}
\Phi_0 = \bigsqcup^4_{i = 1}\Sigma_{P_i}, \
\Phi_1 = \Phi_0 \#_{\bigsqcup_{1 \leq i < j \leq 4} \Sigma_{I_{ij}} \times \del_{in} I_{ij}} \bigsqcup_{1 \leq i < j \leq 4} (\Sigma_{I_{ij}} \times I_{ij}), \
\Phi_2 = \Phi_1 \#_{\del_{in} F} F
\end{align}
where
$\Sigma_{P_i}, \Sigma_{I_{ij}}$
are the fans
$\Sigma_{\bA^2} \subset \bR^2, \Sigma_{\bA^1} \subset \bR$.
We place
the origins of
$\Sigma_{P_i}$
at
$P_i$
and
that of
$\Sigma_{I_{ij}}$
along
${I_{ij}}$
so that
$\Sigma_{P_i} \cap F \subset \Sigma_{P_i}$
and
$\Sigma_{I_{ij}} \cap F \subset \Sigma_{I_{ij}}$.
\end{example}

As explained in
\cite[Section 6]{GS1},
one can generalize fanifolds in terms of stacky fans.
First,
we recall the definition of stacky fans.

\begin{definition}[{\cite[Definition 2.4]{GS15}}]
A
\emph{stacky fan}
is the data of a map of lattices
$\beta \colon \tilde{M} \to M$
with finite cokernel,
together with fans
$\tilde{\Sigma} \subset \tilde{M}_\bR$
and
$\Sigma \subset M_\bR$
such that
$\beta$
induces a combinatorial equivalence on the fans.
\end{definition} 

Stacky fans form a category
$\stFan^\twoheadrightarrow$.
Morphisms
\begin{align*}
(M, \tilde{M}, \Sigma, \tilde{\Sigma})
\to
(M^\prime, \tilde{M}^\prime, \Sigma^\prime, \tilde{\Sigma}^\prime)
\end{align*}
are given by the choices of cones
$\tilde{\sigma}$
with
$\beta(\tilde{\sigma}) = \sigma$
and
compatible isomorphisms
\begin{align*}
(M / \langle \sigma \rangle, \tilde{M} / \langle \tilde{\sigma} \rangle, \Sigma / \sigma, \tilde{\Sigma} / \tilde{\sigma})
\cong
(M^\prime, \tilde{M}^\prime, \Sigma^\prime, \tilde{\Sigma}^\prime).
\end{align*}
Now,
replace the functor
$\Exit(\Phi) \to \Fan^\twoheadrightarrow$
in Definition
\pref{dfn:fanifold}
with
$\Exit(\Phi) \to \stFan^\twoheadrightarrow$.
Defining the
\emph{associated normal fan}
as
$\Sigma_S$
for each value
$(M_S, \tilde{M}_S, \Sigma_S, \tilde{\Sigma}_S)$,
one obtains the generalization.

\section{Weinstein handle attachments}
The Weinstein manifold-with-boundary
$\widetilde{\bfW}(\Phi)$
is obtained by inductively attaching products of cotangent bundles of
real tori
and
strata of
$\Phi$.
Each step requires us to modify Weinstein structures near gluing regions.
In order to show the compatibility of such modifications with candidate maps,
we need to fully understand the attachment process. 

\subsection{Weinstein manifolds}
\begin{definition}[{\cite[Section 11.1]{CE}}, {\cite[Section 1]{Eli}}]
A
\emph{Liouville domain}
$(W, \lambda)$
is a compact symplectic manifold
$(W, \omega = d \lambda)$
with smooth boundary
$\del W$
whose Liouville vector field
$Z = \omega^\# \lambda$
points outwardly along
$\del W$.
We call the positive half
$(\del W \times \bR_{\geq 0}, e^t (\lambda |_{\del W}))$
of the symplectization of
$\del W$
with Liouville form
$e^t (\lambda |_{\del W})$
the
\emph{cylindrical end}.
Any Liouville domain
$(W, \lambda)$
can be completed to a
\emph{Liouville manifold}
$(\widehat{W}, \hat{\lambda})$
by attaching the cylindrical end,
i.e.
$\widehat{W}
=
W \cup (\del W \times \bR_{\geq 0})$,
$\hat{\lambda} |_W
=
\lambda$
and
$\hat{\lambda} |_{\del W \times \bR_{\geq 0}}
=
e^t (\lambda |_{\del W})$.
\end{definition}

\begin{remark}
Strictly speaking,
the completion
$(\widehat{W}, \hat{\lambda})$
is a
\emph{Liouville manifold of finite type}.
Every Liouville manifold of finite type is the completion of some Liouville domain.
Throughout the paper,
following
\cite[Section 4]{GS1},
by a
\emph{Liouville manifold}
we will mean the completion of some Liouville domain.
Note that
the skeleton,
which will be defined below,
of a Liouville manifold of finite type is compact.
\end{remark}

\begin{definition}[{\cite[Section 11.1]{CE}}, {\cite[Definition 4.8]{GS1}}]
Let
$(W, \lambda)$
be a Liouville manifold.
Its
\emph{skeleton}
is the attractor
\begin{align*}
\Core(W)
=
\bigcap_{t > 0} Z^{-t}(W)
\end{align*}
of the negative Liouville flow
$Z^{-t}$.
Equivalently,
$\Core(W)$
is the union of all stable manifolds,
i.e.
the maximal compact subset invariant under the Liouville flow.
We denote by
$([W], [\lambda])$
a Liouville domain
which
completes to
$(W, \lambda)$
and
contains
$\Core(W)$.
The
\emph{ideal boundary}
$\del_\infty W$
of
$W$
is the intersection
$W \cap \del [W]$.
For a subset
$\cL$
of
$\del_\infty W$
the
\emph{relative skeleton}
$\Core(W, \cL)$
of
$(W, \lambda)$
associated with
$\cL$
is the union
\begin{align*}
\Core(W) \cup \bR \cL \subset W
\end{align*}
of
$\Core(W)$
and
the saturation of
$\cL$
by the Liouville flow.
\end{definition}

\begin{definition}[{\cite[Definition 11.10]{CE}}, {\cite[Section 1]{Eli}}]
A
\emph{Weinstein domain}
$(W, \lambda, \phi)$
is a Liouville domain
$(W, \lambda)$
whose Liouville vector field is gradient-like for a Morse--Bott function
$\phi \colon W \to \bR$
which is constant on
$\del W$.
We call the positive half
$(\del W \times \bR_{\geq 0}, e^t (\lambda |_{\del W}), \phi)$
of the symplectization of
$\del W$
with
Liouville form
$e^t (\lambda |_{\del W})$
and
canonically extended
$\phi$
the
\emph{cylindrical end}.
Any Weinstein domain
$(W, \lambda, \phi)$
can be completed to a
\emph{Weinstein manifold}
$(\widehat{W}, \hat{\lambda}, \hat{\phi})$
by attaching the cylindrical end.
\end{definition}

\begin{remark}
The completion
$(\widehat{W}, \hat{\lambda}, \hat{\phi})$
is a
\emph{Weinstein manifold of finite type},
i.e.
$\phi$
has only finitely many critical points.
Every Weinstein manifold of finite type is the completion of some Weinstein domain.
Throughout the paper,
following
\cite[Definition 5.5]{Nad},
by a
\emph{Weinstein manifold}
we will mean the completion of some Weinstein domain.
Note that
in view of
\cite[Example 4.23]{GS1}
one may regard mirror symmetry established in
\cite{GS1}
as a generalization of that in
\cite{GS2},
which relies on
\cite[Theorem 5.13]{Nad}.
\end{remark}

\begin{remark}
For a Weinstein manifold
$(W, \lambda, \phi)$
the skeleton
$\Core(W)$
is isotropic by
\cite[Lemma 11.13(a)]{CE}.
Moreover,
the stable manifold of the critical locus of
$\phi$
contains the zero locus of
$Z$.
\end{remark}

\begin{example}[{\cite[Definition 11.12(2)(3)]{CE}}] \label{eg:Weinstein}
Consider the cotangent bundle
$T^* Y$
of a closed manifold
$Y$
with the standard symplectic form
$\omega_{st} = d \lambda_{st}$
where
$\lambda_{st} = pdq$
is the standard Liouville form.
The associated Liouville vector field
$Z_{st} = p \del_p$
is gradient-like for a Morse--Bott function
$\phi_{st}(q, p) = \frac{1}{2} |p|^2$.
The product 
$T^* Y \times T^* Y^\prime$
of two such Weinstein manifolds with
symplectic form
$\omega_{st} \oplus \omega^\prime_{st}$,
Liouvlle form
$\lambda_{st} \oplus \lambda^\prime_{st}$
and
Morse--Bott function
$\phi_{st} \oplus \phi^\prime_{st}$ 
is a Weinstein manifold.
\end{example}

\subsection{Weinstein pairs}
\begin{definition}[{\cite[Section 2]{Eli}}] 
Let
$(Y, \xi)$
be a contact manifold.
We call a codimension
$1$
submanifold
$H \subset Y$
with smooth boundary a
\emph{Weinstein hypersurface}
if there exists a contact form
$\lambda$
for
$\xi$
such that
$(H, \lambda |_H)$
is compatible with a Weinstein structure on
$H$,
i.e.,
$\omega_H = d \lambda |_H$
is a symplectic form,
$Z_H = \omega^\#_H \lambda |_H$
points outwardly along
$\del H$
and
is gradient-like for some Morse--Bott function
$\phi_H \colon H \to \bR$.
Its
\emph{contact surrounding}
$U_\epsilon(H)$
is the neighborhood of
$H$
in
$Y$
defined as follows.
Let
$\tilde{H}$
be a slightly extended Weinstein hypersurface
satisfying
$\lambda |_{\tilde{H} \setminus H} = t \lambda |_{\del H}$
for
$t \in [1, 1+\epsilon]$.
There is a neighborhood
$\tilde{U}$
of
$\tilde{H}$
diffeomorphic to
$\tilde{H} \times (-\epsilon, \epsilon)$
with
$\lambda |_{\tilde{U}}
=
\pr^*_1 (\lambda |_{\tilde{H}}) + du$
where
$u$
is the coordinate of the second factor.
For a nonnegative function
$h \colon \tilde{H} \to \bR$
which equals
$0$
on
$H$
and
$t-1$
near
$\del \tilde{H}$,
set
$U_\epsilon(H)
=
\{ h^2 + u^2 \leq \epsilon^2 \}
\subset
\tilde{U}$.
\end{definition}

\begin{remark}
Although the induced Weinstein structure on
$H$
depends on the choice of
$\lambda$,
the skeleton
$\Core(H)$
is independent of the choice
\cite[Lemma 12.1]{CE}.
\end{remark}

\begin{example}[{\cite[Example 2.1(i)]{Eli}}] \label{eg:thickening}
Let
$\cL$
be a Legendrian submanifold of a contact manifold
$(Y, \xi)$.
It admits a neighborhood
$U(\cL)$
isomorphic to
\begin{align*}
(J^1(\cL), du - pdq), \ q \in \cL, \ \| p \|^2 + u^2 \leq \epsilon.
\end{align*}
Then
$H(\cL) = U(\cL) \cap \{ u = 0 \}$
is a Weinstein hypersuface called a
\emph{Weinstein thickening}
of
$\cL$.
It is symplectomorphic to the cotangent ball bundle of
$\cL$.
Up to Weinstein isotopy,
$H(\cL)$
is independent of all the choices.
\end{example}

\begin{definition}[{\cite[Section 2]{Eli}}] 
A
\emph{Weinstein pair}
$(W, H)$
consists of a Weinstein domain
$(W, \lambda, \phi)$
together with a Weinstein hypersurface
$(H, \lambda |_H)$
in
$\del W$.
The
\emph{skeleton}
$\Core(W, H)$
of
$(W, H)$
is the relative skeleton of the Liouville manifold
$(W, \lambda)$
associated with
$H$.
\end{definition}


\begin{definition}[{\cite[Section 2]{Eli}}]
Let
$(W, H)$
be a Weinstein pair,
$\phi_H \colon H \to \bR$
a Morse--Bott function
for which
$Z_H = (\omega |_H)^\# \lambda |_H$
is gradient-like
and
$U_\epsilon(H) \subset \del W$
the contact surrounding of
$H$.
We call a pair
$(\lambda_0, \phi_0)$
of a Liouville form
$\lambda_0$
for
$\omega$
and
a smooth function
$\phi_0 \colon W \to \bR$
\emph{adjusted}
to
$(W, H)$
if
\begin{itemize}
\item
$Z_0 = \omega^\# \lambda_0$
is tangent to
$\del W$
on
$U_\epsilon(H)$
and
transverse to
$\del W$
elsewhere;
\item
$Z_0 |_{U_\epsilon(H)} = Z_H + u \del_u$;
\item
the attractor
$\bigcap_{t > 0} Z^{-t}_0(W)$
coincides with
$\Core(W, H)$;
\item
the function
$\phi_0$
is Morse--Bott for which
$Z_0 = (\omega)^\# \lambda_0$
is gradient-like
satisfying
$\phi_0 |_{U_\epsilon(H)} = \phi_H + \frac{1}{2} u^2$
and
whose critical values are not more than
$\phi_0 |_{\del U_\epsilon(H)}$.
\end{itemize}
\end{definition}

Given a Weinstein pair
$(W, H)$,
one can always modify
the Liouville form
$\lambda$
for
$\omega$
and
the Morse--Bott function
$\phi_H$
to be adjusted
while preserving
$\lambda$
in the complement of
$U_\epsilon(H)$.

\begin{lemma}[{\cite[Proposition 2.9]{Eli}}] \label{lem:modify}
Let
$(W, H)$
be a Weinstein pair.
Then there exist
a Liouville form
$\lambda_0$
for
$\omega$
and
a smooth function
$\phi_0 \colon W \to \bR$
such that
$(\lambda_0, \phi_0)$
are adjusted to
$(W, H)$
and
$\lambda_0 |_{W \setminus U_\epsilon(H)}
=
\lambda |_{W \setminus U_\epsilon(H)}$.
\end{lemma}

\subsection{Gluing of Weinstein pairs}
\begin{definition}[{\cite[Section 3.1]{Eli}}] 
Let
$(W, \lambda, \phi)$
be a Weinstein domain.
A
\emph{splitting}
for
$(W, \lambda, \phi)$
is a hypersurface
$(P, \del P) \subset (W, \del W)$
satisfying the following conditions.
\begin{itemize}
\item
$\del P$
and
$P$
respectively split
$\del W$
and
$W$
into two parts
$\del W = Y_- \cup Y_+$
with
$\del Y_- = \del Y_+ = Y_- \cap Y_+ = \del P$
and
$W = W_- \cup W_+$
with
$\del W_- = P \cup Y_-,
\del W_+ = P \cup Y_+,
W_- \cap W_+ = P$.
\item
The Liouville vector field
$Z = \omega^\# \lambda$
is tangent to
$P$.
\item
There exists a hypersurface
$(H, \del H) \subset (P, \del P)$
which is Weinstein for
$\lambda |_H$,
tangent to
$Z$
and
intersects transversely all leaves of the characteristic foliation of
$P$. 
\end{itemize}
\end{definition}

The above hypersurface
$(H, \del H)$
is called the
\emph{Weinstein soul}
for the splitting hypersurface
$P$.
Due to
\pref{lem:locCont}
below,
$P$
is contactomorphic to the contact surrounding
$U_\epsilon(H)$
of its Weinstein soul.

\begin{definition}[{\cite[Section 2]{Eli}}] 
Let
$H$
be a closed hypersurface in a $(2n-1)$-dimensional manifold
and
$\xi$
a germ of a contact structure along
$H$
which admits a transverse contact vector field
$V$.
The
\emph{invariant extension}
of the germ
$\xi$
is the canonical extension
$\hat{\xi}$
on
$H \times \bR$,
which is invariant with respect to translations along the second factor
and
whose germ along any slice
$H \times \{ t \}, t \in \bR$
is isomorphic to
$\xi$.
\end{definition}

\begin{lemma}[{\cite[Lemma 2.6]{Eli}}] \label{lem:locCont}
Let
$H$
be a closed $(2n-2)$-dimensional manifold
and
$\xi$
a contact structure on
$P = H \times [0, \infty)$
which admits a contact vector field
$V$
inward transverse to
$H \times \{ 0 \}$ 
such that
its trajectories intersecting
$H \times \{ 0 \}$
fill the whole manifold
$P$.
Then
$(P, \xi)$
is contactomorphic to
$(H \times [0, \infty), \hat{\xi})$,
where
$\hat{\xi}$
is the invariant extension of the germ of
$\xi$
along
$H \times \{ 0 \}$.
Moreover,
for any compact set
$C \subset P$
with
$H \times \{ 0 \} \subset C$
the contactomorphism can be taken
so that
it is equals to the identity on
$H \times \{ 0 \}$
and
sends
$V |_C$
to the vector field
$\del_t$.
\end{lemma}

\begin{definition}[{\cite[Section 3.1]{Eli}}] \label{dfn:gluing}
Let
$(W, H), (W^\prime, H^\prime)$
be Weinstein pairs
with adjusted
Liouville forms
and
Morse--Bott functions
$(\lambda, \phi), (\lambda^\prime, \phi^\prime)$.
Suppose that
there is an isomorphism
$(H, \lambda |_H, \phi |_H)
\cong
(H^\prime, \lambda^\prime |_{H^\prime}, \phi^\prime |_{H^\prime})$
of Weinstein manifolds.
Extend it to their contact surroundings
$U_\epsilon(H) \cong U_\epsilon(H^\prime)$.
Then the
\emph{gluing}
of 
$(W, H), (W^\prime, H^\prime)$
is given by
$W \#_{H \cong H^\prime} W^\prime
=
W \cup_{U_\epsilon(H) \cong U_\epsilon(H^\prime)} W^\prime$
with glued
Liouville form
and
Morse--Bott function.
\end{definition}

By definition we have
\begin{align*}
\Core(W \#_{H \cong H^\prime} W^\prime)
=
\Core(W, H) \cup_{\Core(H) \cong \Core(H^\prime)} \Core(W^\prime, H^\prime).
\end{align*}

\subsection{Weinstein handle attachments}
Let
$W$
be a Weinstein domain with a smooth Legendrian
$\cL \subset \del_\infty W$.
Fix a standard neighborhood of
$\cL$
in
$\del_\infty W$
\begin{align*}
\eta
\colon
\Nbd_{\del_\infty W}(\cL)
\hookrightarrow
J^1 \cL
=
T^* \cL \times \bR
\end{align*}
which extends to a neighborhood in the Weinstein domain
$W$
\begin{align*}
\xi
\colon
\Nbd_W(\cL)
\hookrightarrow
J^1 \cL \times \bR_{\leq 0}
\cong
T^*(\cL \times \bR_{\leq 0}),
\end{align*}
where the Liouville flow on
$W$
gets identified with the translation action on
$\bR_{\leq 0}$.
Note that
$\eta^{-1}(T^* \cL \times \{ 0 \})$
gives a Weinstein thichkening of
$\cL$.

Due to
\pref{lem:modify},
one can modify the Weinstein structure near
$\cL$
so that
the Liouville flow gets identified with the cotangent scaling on
$T^*(\cL \times \bR_{\leq 0})$.
In other words,
the Liouville form
and
the Morse--Bott function
become adjusted to the Weinstein pair
$(W, \eta^{-1}(T^* \cL \times \{ 0 \}))$.
We denote by
$\widetilde{W}$
its completion.
Then
$\eta$
yields a neighborhood of
$\cL$
in
$\del_\infty \widetilde{W}$ 
\begin{align*}
\widetilde{\eta}
\colon
\Nbd_{\del_\infty \widetilde{W}}(\cL)
\hookrightarrow
J^1 \cL
\end{align*}
which extends to a neighborhood in
$\widetilde{W}$
\begin{align*}
\widetilde{\xi}
\colon
\Nbd_{\widetilde{W}}(\cL)
\hookrightarrow
J^1 \cL \times \bR
\cong
T^*(\cL \times \bR).
\end{align*}

\begin{remark}
As a Liouville domain,
the modification is canonical up to contractible choice,
since any two Liouville structures on a compact symplectic manifold are canonically homotopic
\cite[page 2]{Eli}.
\end{remark}

\begin{remark} \label{rmk:sector}
The result
$\widetilde{W}$
is a Liouville sector in the sense of
\cite[Definition 2.4]{GPS1}.
Indeed,
$\widetilde{W}$
is a Liouville manifold-with-boundary,
whose boundary consists of
$\Nbd_{\del_\infty \widetilde{W}}(\cL)$
and
the completion of
$\del \Nbd_{\del_\infty W}(\cL)$.
The characteristic foliations of
the
former 
and
latter
are respectively given by
the Reeb vector field on
$\del W$
and
the restricted Liouville vector field on
$\widetilde{W}$
up to negation.
Now,
it is clear that
$\del \del_\infty \widetilde{W}$
is convex
and
there is a diffeomorphism
$\del \widetilde{W} \cong \bR \times F$
sending the characteristic foliation of
$\del \widetilde{W}$
to the foliation of
$\bR \times F$
by leaves
$\bR \times \{ p \}$.
\end{remark}

Given Weinstein domains
$W, W^\prime$
with smooth Legendrian embeddings
$\del_\infty W \hookleftarrow \cL \hookrightarrow \del_\infty W^\prime$,
we write
$W \#_\cL W^\prime$
for the gluing
\begin{align*}
W \#_{\eta^{-1}(T^* \cL \times \{ 0 \}) \cong \eta^{\prime -1}(T^* \cL \times \{ 0 \})} W^\prime
=
\widetilde{W} \cup_{J^1 \cL} \widetilde{W^\prime}
\end{align*}
which yields a Weinstein manifold-with-boundary with skeleton
\begin{align*}
\Core(W \#_\cL W^\prime)
=
\Core(W, \cL)
\cup_\cL
\Core(W^\prime, \cL).
\end{align*}

\begin{definition}[{\cite[Definition 4.9]{GS1}}] \label{dfn:biconicity}
Let
$W$
be a Weinstein domain with a smooth Legendrian
$\cL \subset \del_\infty W$.
A conic subset
$\bL \subset W$
is
\emph{biconic}
along
$\cL$
if the image of
$\bL \cap \Nbd_W(\cL)$
under some
$\xi
\colon
\Nbd_W(\cL)
\hookrightarrow
T^*(\cL \times \bR_{\leq 0})$,
where the Liouville flow on
$W$
gets identified with the translation action on
$\bR_{\leq 0}$,
is invariant also under the cotangent scaling of
$T^*(\cL \times \bR_{\leq 0})$.
\end{definition}

By construction any biconic subset
$\bL \subset W$
remains conic in
$\widetilde{W}$.
We write
$\widetilde{\bL} \subset \widetilde{W}$
for its saturation under the Liouville flow.

\begin{lemma}[{\cite[Lemma 4.10]{GS1}}] \label{lem:biconicity}
Let
$W$
be a Weinstein domain with a smooth Legendrian
$\cL \subset \del_\infty W$.
Then a Lagrangian
$\bL \subset W$
is biconic along
$\cL$
if and only if
it is conic
and
\begin{align*}
\xi(\bL \cap \Nbd_W(\cL))
\subset
T^* \cL \times \{ 0 \} \times \bR_{\leq 0}
\subset
J^1 \cL \times \bR_{\leq 0}
\cong
T^*(\cL \times \bR_{\leq 0})
\end{align*}
for some
$\xi
\colon
\Nbd_W(\cL)
\hookrightarrow
T^*(\cL \times \bR_{\leq 0})$,
where the Liouville flow on
$W$
gets identified with the translation action on
$\bR_{\leq 0}$.
\end{lemma}
\begin{proof}
The second factor
$\bR$
in the product
$T^* \cL \times \bR = J^1 \cL$
is responsible for the additional cotangent scaling of
$T^*(\cL \times \bR_{\leq 0})$
with respect to
$T^* \cL$.
\end{proof}

Given biconic Lagrangians
$\bL \subset W, \bL^\prime \subset W^\prime$
with matching ends in the sense that
\begin{align*}
\widetilde{\eta}(\bL \cap \Nbd_{\del_\infty W}(\cL))
=
\widetilde{\eta}^\prime(\bL^\prime \cap \Nbd_{\del_\infty W^\prime}(\cL))
\end{align*}
in
$W \#_\cL W^\prime$,
we write
$\bL \#_\cL \bL^\prime$
for the gluing
$\widetilde{\bL}
\cup_{\widetilde{\eta}(\bL \cap \Nbd_{\del_\infty W}(\cL))}
\widetilde{\bL^\prime}$
in
$W \#_\cL W^\prime$.
Since any biconic subsets in
$W, W^\prime$
remain conic in
$W \#_\cL W^\prime$,
the gluing
$\bL \#_\cL \bL^\prime$
is a conic Lagrangian.

For a closed manifold
$\widehat{M}$
and
a manifold-with-boundary
$S$,
consider the Weinstein domain-with-boundary
$W^\prime = [T^* \widehat{M} \times T^* S]$
with a smooth Legendrian
$\cL = \widehat{M} \times \del S$
taken to be a subset of the zero section.
The Liouville flow on
$W^\prime$
near
$\cL$
is the cotangent scaling.
We write
$\widetilde{W^\prime}$
for the completion
$T^* \widehat{M} \times T^* S$,
which is a Weinstein manifold-with-boundary.
One can check that
the above gluing procedure carries over,
although
$\cL$
does not belong to
$\del_\infty W^\prime$.

\begin{remark}
As explained in
\cite[Example 2.3]{GPS1},
when
$S$
is noncompact
$T^* S$
equipped with the standard Liouville form is not even a Liouville manifold.
However,
if
$S$
is the interior of a compact manifold-with-boundary,
then
$T^* S$
becomes a Liouville manifold after a suitable modification of the Liouville form near the boundary making it convex.
We
use the same symbol
$T^* S$
to denote the result
and
regard it as a stopped Weinstein sector.
\end{remark}

\begin{definition}[{\cite[Definition 4.11]{GS1}}] \label{dfn:extension}
Let
$W$
be a Weinstein domain with a smooth Legendrian embedding
$\widehat{M} \times \del S \hookrightarrow \del_\infty W$.

(1)
A
\emph{handle attachment}
is the gluing
\begin{align*}
W \#_{\widehat{M} \times \del S} [T^*\widehat{M} \times T^* S]
\end{align*}
respecting the product structure.

(2)
Let further
$\bL \subset W$
be a biconic subset along
$\widehat{M} \times \del S$
which locally factors as
\begin{align*}
\eta(\bL \cap \Nbd_{\del_\infty W}(\widehat{M} \times \del S))
=
\bL_S \times \del S
\subset
T^*\widehat{M} \times T^* \del S
\end{align*}
for some fixed conic Lagrangian
$\bL_S$.
The
\emph{extension}
of
$\bL$
through the handle is the gluing
\begin{align*}
\bL \#_{\widehat{M} \times \del S} (\bL_S \times S)
\end{align*}
respecting the product structure.

(3)
Let further
$\Lambda \subset \del_\infty W$
be a subset satisfying
\begin{align*}
\eta(\Lambda \cap \Nbd_{\del_\infty W}(\widehat{M} \times \del S))
\cong
\bL_S \times \del S
\subset
T^*\widehat{M} \times T^* \del S \times \{ 0 \}
\subset
J^1 (\widehat{M} \times \del S)
\end{align*}
for
$\bL_S$.
The
\emph{extension}
of
$\Lambda$
through the handle is the ideal boundary
\begin{align*}
\del_\infty (\Core(W, \Lambda) \#_{\widehat{M} \times S} (\bL_S \times S))
\end{align*}
of the extension of the biconic Lagrangian
$\Core(W, \Lambda)$
thorough the handle.
\end{definition}

\begin{lemma}[{\cite[Lemma 4.12]{GS1}}] \label{lem:biconicity-criterion}
Let
$E \to Y$
be a vector bundle.
Then near the Legendrian
$\cL = \del_\infty T^*_Y E \subset T^* E$
there are local coordinates
such that
if
$\Lambda \subset \bigcup_{\alpha} T^*_{S_\alpha} E$
for any collection
$\{ S_\alpha \}_\alpha$
of conic subsets
$S_\alpha \subset E$
with respect to the scaling of
$E$,
then
$\Lambda$
is biconic along
$\cL$.
\end{lemma}
\begin{proof}
As a bundle over
$Y$,
we have
$T^* E = E \oplus E^\vee \oplus T^* Y$.
Let
$P \subset T^* E$
be the polar hypersurface defined as the kernel of the pairing between
$E$
and
$E^\vee$.
Then
$P$
contains
$T^*_{S_\alpha} E$
for any conic subset
$S_\alpha \subset E$.
Indeed,
a point
$(y,u,v,w) \in T^*_y E$
belongs to the fiber
$(T^*_{S_\alpha} E)_y \subset (T^* E |_{S_\alpha})_y$
only if
$(y,u) \in S_\alpha$
and
$v(u) = 0$.
Locally,
the ideal boundary
$\del_\infty P \subset \del_\infty T^* E$
can be identified with
$T^* \cL$
compatibly with their standard Liouville structures.
Indeed,
any point of
$(\del_\infty P)_y$
is expressed as
$(y,u,[v],w)$
with
$(y,u,v,w) \in P_y$,
while any point of
$T^*_y \cL$
is expressed as
$(y,[v],w,u)$
with
$(y,[v]) \in (\del_\infty T^*_Y E)_y$,
$w \in T^*Y_y$,
and
$u \in E_y$.
Here,
we identify
$E_y$
with the hyperplane orthogonal to
$v$
passing through the origin.
Hence locally
$\del_\infty P$
is a Weinstein thickening of
$\cL$
from Example
\pref{eg:thickening}.
Transporting to
$\Nbd_P(\cL)$
the Reeb vector flow
which intersects
$T^* \cL$
transversely,
we obtain local coordinates
$\eta \colon \Nbd_P(\cL) \hookrightarrow J^1 \cL$
whose biconicity follows from
\pref{lem:biconicity}.
\end{proof}

\begin{corollary}[{\cite[Corollary 4.13]{GS1}}] \label{cor:biconicity-criterion}
Let
$Y \subset Y^\prime$
be a submanifold.
Then near the Legendrian
$\del_\infty T^*_Y Y^\prime$
there are local coordinates
such that
if
$\Lambda \subset \bigcup_{\alpha} T^*_{S_\alpha} Y^\prime$
for any collection
$\{ S_\alpha \}_\alpha$
of subsets
$S_\alpha \subset Y$,
then
$\Lambda$
is biconic.
\end{corollary}
\begin{proof}
Locally identify a tubular neighborhood
$N_Y Y^\prime$
of
$Y$
in
$Y^\prime$
with
$Y^\prime$.
Since any subset
$S_\alpha \subset Y$
trivially becomes conic with respect to the scaling of
$N_Y Y^\prime$,
one can apply
\pref{lem:biconicity-criterion}
for
$E = N_Y Y^\prime$
and
$\cL = \del_\infty T^*_Y N_Y Y^\prime$
to obtain the desired local coordinates.
\end{proof}

\begin{definition}[{\cite[Section 3.1]{FLTZ12}}]
Let
$\Sigma \subset M_\bR$
be a rational polyhedral fan
and
$\widehat{M}$
the real $n$-torus
$\Hom(M, \bR / \bZ) = M^\vee_\bR / M^\vee$.
The
\emph{FLTZ Lagrangian}
is the union
\begin{align*}
\bL(\Sigma)
=
\bigcup_{\sigma \in \Sigma} \bL_\sigma
=
\bigcup_{\sigma \in \Sigma} \sigma^\perp \times \sigma
\end{align*}
of conic Lagrangians
$\sigma^\perp \times \sigma
\subset
\widehat{M} \times M_\bR
\cong
T^*\widehat{M}$,
where
$\sigma^\perp$
is the real $(n - \dim \sigma)$-subtorus
\begin{align*}
\{ x \in \widehat{M} \ | \ \langle x, v \rangle = 0 \ \text{for all } v \in \sigma \}.
\end{align*}
\end{definition}

\begin{remark}[{\cite[Definition 6.3]{FLTZ14}}]
When
$\Sigma$
is a stacky fan,
the FLTZ Lagrangian
$\bL(\Sigma)$
becomes the union
$\bigcup_{\sigma \in \Sigma} G_\sigma \times \sigma$,
where
$G_\sigma = \Hom(M_\sigma, \bR / \bZ)$
are possibly disconnected subgroup of
$\widehat{M}$.
Here,
we denote by
$M_\sigma$
the quotients of
$M$
by the stacky primitives for
$\sigma$. 
\end{remark}

\begin{lemma}[{\cite[Lemma 4.16]{GS1}}] \label{lem:coordinates}
Let
$\Sigma \subset M_\bR$
be a rational polyhedral fan.
Then for each cone
$\sigma \in \Sigma$
near the Legendrian
$\del_\infty \bL_\sigma \subset \del_\infty T^* \widehat{M}$
there are local coordinates
\begin{align*}
\eta_\sigma
\colon
\Nbd(\del_\infty \bL_\sigma)
\hookrightarrow
J^1 \del_\infty \bL_\sigma = T^* \del_\infty \bL_\sigma \times \bR 
\end{align*}
with the following properties.
\begin{itemize}
\item[(1)]
The Lagranzian
$\bL_\sigma$
is biconic along
$\del_\infty \bL_\sigma$.
\item[(2)]
For any nonzero cones
$\sigma, \tau \in \Sigma$
with
$\sigma \subset \bar{\tau}$,
we have
\begin{align} \label{eq:recursive}
\eta_\sigma (\del_\infty \bL_\tau \cap \Nbd(\del_\infty \bL_\sigma))
=
\bL_{\tau / \sigma} \times \del_\infty \sigma \times \{ 0 \}
\subset
T^* \sigma^\perp \times T^* \del_\infty \sigma \times \{ 0 \} 
=
T^* \del_\infty \bL_\sigma \times \{ 0 \}.
\end{align}
\item[(3)]
The FLTZ Lagranzian
$\bL(\Sigma)$
is biconic along each
$\del_\infty \bL_\sigma$.
\item[(4)]
For each
$\sigma \in \Sigma$
the local coordinates
$\eta_\sigma$
define a Weinstein hypersuface
\begin{align*}
R_\sigma = \eta^{-1}_\sigma (T^* \del_\infty \bL_\sigma \times \{ 0 \})
\subset
\del_\infty T^* \widehat{M}
\end{align*}
containing
$\del_\infty \bL_\sigma$
as its skeleton
and
we have
$R_\tau \cap \Nbd(\del_\infty \bL_\sigma) \subset R_\sigma$
for any
$\tau \in \Sigma$
with
$\sigma \subset \bar{\tau}$.
\end{itemize}
\end{lemma}
\begin{proof}
Applying Corollary
\pref{cor:biconicity-criterion}
to each submanifold
$\sigma^\perp \subset \widehat{M}$,
we obtain local coordinates
$\eta_\sigma \colon \Nbd(\del_\infty \bL_\sigma) \hookrightarrow J^1 \del_\infty \bL_\sigma$
near
$\del_\infty \bL_\sigma$
such that
$\bL_\sigma$
is biconic along
$\del_\infty \bL_\sigma$.
Provided
(1),
the property
(3)
follows from
(2),
which is a consequence of the identification
\begin{align*}
\del_\infty \bL_\tau
\cong
\tau^\perp \times \del_\infty(\tau / \langle \sigma \rangle \times \sigma)
\cong
(\tau / \langle \sigma \rangle)^\perp \times \tau / \langle \sigma \rangle  \times \del_\infty \sigma
\subset
\bL(\Sigma / \sigma) \times \del_\infty \sigma
\end{align*}
near
$\sigma$
of subsets of the stratified spaces
$\Sigma$
and
$\Sigma / \sigma$
for any
$\tau \in \Sigma$
with
$\{ 0 \} \neq \sigma \subset \bar{\tau}$.
Note that
subtori
$(\tau / \langle \sigma \rangle)^\perp, \tau^\perp$
of respective tori
$\widehat{M / \langle \sigma \rangle} \subset \widehat{M}$
are the same.
When
$\sigma = \{ 0 \}$
and
taking the ideal boundary
$\del_\infty \sigma$
makes no sense,
we have instead
\begin{align*}
\eta_\sigma (\del_\infty \bL_\tau \cap \Nbd(\del_\infty \bL_\sigma))
=
\del_\infty \bL_{\tau / \sigma} \times \{ 0 \}
\subset
T^* \sigma^\perp \times \{ 0 \}.
\end{align*}

It remains to show
(4).
Let
$P_\sigma$
be the polar hypersurface defined as the kernel of the pairing
\begin{align*}
f_\sigma
\colon
T^* U_\sigma
\cong
T^* N_{\sigma^\perp} \widehat{M}
=
N_{\sigma^\perp} \widehat{M}
\oplus
(N_{\sigma^\perp} \widehat{M})^\vee
\oplus
T^* \sigma^\perp
\to
\bR,
\end{align*}
which gives rise to the above local coordinates
$\Nbd_{P_\sigma}(\del_\infty T^*_{\sigma^\perp} \widehat{M})
\hookrightarrow
J^1 \del_\infty T^*_{\sigma^\perp} \widehat{M}$
as in the proof of
\pref{lem:biconicity-criterion}.
However,
there is no inclusion between
$\del_\infty P_\sigma, \del_\infty P_\tau$
for
$\sigma \subset \bar{\tau}$.
To remedy this issue,
one needs to modify the above local coordinates as follows.
For a $1$-dimensional cone
$\sigma$,
replace
$f_\sigma$
with
$g_\sigma = f_\sigma |_{V_\sigma}$
where
$V_\sigma \cong N_{\bL_\sigma} T^*\widehat{M}$
is a conic tubular neighborhood projecting to
$U_\sigma$
and
$W_\sigma \cong N_\sigma M_\bR$
via the identification
$T^* \widehat{M} \cong \widehat{M} \times M_\bR$,
chosen so that
$V_\sigma \cap V_{\sigma^\prime}$
when
$\bar{\sigma} \cap \bar{\sigma}^\prime = \{ 0 \}$.
Locally, we have
\begin{align*}
R_\sigma
=
g^{-1}_\sigma(0)
\cap
T^* \del_\infty T^*_{\sigma^\perp} \widehat{M} \times \{ 0 \}
\cong
T^* \del_\infty T^*_{\sigma^\perp} \widehat{M}
\end{align*}
and
the modified local coordinates
\begin{align*}
\Nbd_{g^{-1}_\sigma(0)}(\del_\infty T^*_{\sigma^\perp} \widehat{M})
\hookrightarrow
J^1 \del_\infty T^*_{\sigma^\perp} \widehat{M}
\end{align*}
satisfies the properties
(1), (2), (3).

For a higher dimensional cone
$\tau$,
one extends
$g_\sigma$
as follows.
When
$\dim \tau = 2$
and
$\tau$
is spanned by
$\sigma_1, \sigma_2$,
replace
$f_\tau$
with
\begin{align*}
g_\tau
=
a_1 f_{\sigma_1} |_{V_\tau}
+
a_2 f_{\sigma_2} |_{V_\tau}
\colon
V_\tau
\to
\bR
\end{align*}
where
$(a_1, a_2) \colon V_\tau \to \tau = \langle \sigma_1, \sigma_2 \rangle$
is the canonical projection.
The functions
$a_i$
take values in
$\bR_{\geq 0}$
and
satisfy
$a_1 |_{V_{\sigma_2} \cap V_\tau}
=
a_2 |_{V_{\sigma_1} \cap V_\tau}
=
0$.
Since we have
$(N_{\tau^\perp} \widehat{M})^\vee
\supset
N_\tau M_\bR
\subset
N_{\sigma_i} M_\bR
\subset
(N_{\sigma^\perp_i} \widehat{M})^\vee$
as a bundle over
$\tau^\perp$,
the restrictions
$(f_{\sigma_i} |_{V_\tau}) |_{V_{\sigma_i} \cap V_\tau}$
coincide with
$(g_{\sigma_i}) |_{V_{\sigma_i} \cap V_\tau}$.
Hence
$g_\tau$
is an extension of
$(g_{\sigma_1}, g_{\sigma_2})
\colon
V_{\sigma_1} \sqcup V_{\sigma_2}
\to
\bR$.
Locally,
we have
\begin{align*}
R_\tau
=
g^{-1}_\tau(0)
\cap
T^* \del_\infty T^*_{\tau^\perp} \widehat{M} \times \{ 0 \}
\cong
T^* \del_\infty T^*_{\tau^\perp} \widehat{M}
\end{align*}
and
the modified local coordinates
\begin{align*}
\Nbd_{g^{-1}_\tau(0)}(\del_\infty T^*_{\tau^\perp} \widehat{M})
\hookrightarrow
J^1 \del_\infty T^*_{\tau^\perp} \widehat{M}
\end{align*}
satisfies the properties
(1), (2), (3)
and
\begin{align*}
R_\tau \cap \Nbd_{g^{-1}_{\sigma_i}(0)}(\del_\infty T^*_{\sigma^\perp_i} \widehat{M}) \subset R_{\sigma_i}.
\end{align*}
Inductively,
one obtains the desired extensions for the other cases.
\end{proof}


\begin{remark}
Some readers of
\cite{GS1}
might wonder why the authors stated
\pref{lem:coordinates}(4).
As explained in
Section 3.3,
we glue two Weinstein domains along their isomorphic Weinstein hypersurfaces.
Here,
each Weinstein hypersurface
$R_\sigma$
is
defined in the modified polar hypersurface
$g^{-1}_\sigma(0)$
and
coincides with a Weinstein thickening of
$\del_\infty \bL_\sigma$.
\pref{lem:coordinates}(4)
guarantees compatibility of such local Weinstein pairs with the union
$\bL(\Sigma) = \bigcup_{\sigma \in \Sigma}\bL_\sigma$.
\end{remark}

\section{The proof of \pref{thm:fibration} explained}
In this section,
we review the proof of
\pref{thm:fibration},
filling in the details.
By
\pref{lem:filtration}
the fanifold
$\Phi$
of dimension
$n$
admits a filtration
\pref{eq:filtration1}.
We proceed by induction on
$k$.

\subsection{Base case} 
When
$k = 0$,
let
\begin{align*}
\widetilde{\bfW}(\Phi_0)
=
\bigsqcup_P T^* \widehat{M}_P, \
\widetilde{\bL}(\Phi_0)
=
\bigsqcup_P \bL(\Sigma_P)
\end{align*}
where
$P$
runs through all $0$-strata of
$\Phi$
and
each
$T^* \widehat{M}_P$
is equipped with the canonical Liouville structure from Example
\pref{eg:Weinstein}.
We fix an identification
$T^* \widehat{M}_P \cong \widehat{M}_P \times M_{P, \bR}$
to regard each
$\bL(\Sigma_P)$
as a conic Lagrangian submanifold of
$T^* \widehat{M}_P$.

\begin{lemma} \label{lem:0-W-subanalytic}
The manifold
$\widetilde{\bfW}(\Phi_0)$
is subanalytic Weinstein.
\end{lemma}
\begin{proof}
Equipped with the canonical Liouville structure,
$T^* \widehat{M}_P$
become Weinstein as explained in Example
\pref{eg:Weinstein}.
In general,
real analytic subsets of real analytic manifolds are subanalytic
and
the product of subanalytic subsets of real analytic manifolds is subanalytic. 
\end{proof}

\begin{lemma}
The Lagrangian
$\widetilde{\bL}(\Phi_0) \subset \widetilde{\bfW}(\Phi_0)$
is subanalytic conic
and
contains
$\Core(\widetilde{\bfW}(\Phi_0))$.
\end{lemma}
\begin{proof}
The skeleton
$\Core(\widetilde{\bfW}(\Phi_0))$
is the disjoint union
$\bigsqcup_P \widehat{M}_P$
of the zero sections
$\widehat{M}_P \subset T^* \widehat{M}_P$
for all $0$-strata
$P$
of
$\Phi$,
each of which is contained in the conic Lagrangian
$\bL(\Sigma_P)$
via the fixed identification
$T^* \widehat{M}_P \cong \widehat{M}_P \times M_{P, \bR}$.
In general,
the collection of subanalytic subsets of a real analytic manifold forms a Boolean algebra.
Hence
$\bL(\Sigma_P)$
is subanalytic,
as it is the union of products of
a real torus
and
an algebraic subset of a Euclidean space,
which are subanalytic subsets of
$\widehat{M}_P \times M_{P, \bR}$.
\end{proof}

Let
$\pi_0 \colon \widetilde{\bL}(\Phi_0) \to \Phi_0$
be the map induced by the projection to cotangent fibers.

\begin{lemma} \label{lem:0-original}
The triple
$(\widetilde{\bfW}(\Phi_0), \widetilde{\bL}(\Phi_0), \pi_0)$
satisfies the conditions
(1), \ldots, (4).
\end{lemma}
\begin{proof}
(1)
It suffices to show the claim for a neighborhood
$\Nbd(P)$
of a single $0$-stratum
$P \subset \Phi$.
Let
$S_P \subset \Phi_0$
be the stratum of codimension
$d$
corresponding to a cone
$\sigma_P \in \Sigma_P$
via
$\phi_P$.
Then by definition of
$\pi_0$
we have
\begin{align*}
\pi^{-1}_0(S_P)
=
\bL_{\sigma_P}
=
\sigma^\perp_P \times \sigma_P
\cong
T^d \times S_P.
\end{align*}

From the isomorphism
\begin{align} \label{eq:0-decomposition}
\bL_{\tau_P}
=
\tau^\perp_P \times \tau_P
\cong
(\tau_P / \langle \sigma_P \rangle)^\perp \times \tau_P / \langle \sigma_P \rangle \times \sigma_P
=
\bL_{\tau_P / \sigma_P} \times \sigma_P
\end{align}
for any stratum
$\tau_P \in \Sigma_P$
with
$\sigma_P \subset \bar{\tau}_P$
it follows
\begin{align*}
\pi^{-1}_0(\Nbd(S_P))
=
\bigcup_{\tau_P \in \Sigma_P, \ \sigma_P \subset \bar{\tau}_P} \bL_{\tau_P}
\cong
\bL(\Sigma_P / \sigma_P) \times S_P.
\end{align*}

Consider the cotangent bundle
$T^* T^d \times T^* S_P$
with the standard symplectic form.
Since we have
$T^* S_P
\cong
S_P \times \bR^{\dim \widehat{M}_P - d}
\cong
S_P \times S_P$,
there is a symplectomorphism
\begin{align}  \label{eq:0-foliation}
T^* T^d \times T^* S_P
\cong
T^* T^d \times T^* S_P, \
((\theta, \eta), (x, y))
\mapsto
((\theta, \eta), (-y, x)).
\end{align}
Then an open embedding
\begin{align*}
T^d \times T^*_x S_P
\hookrightarrow
T^d \times T^{\dim \widehat{M}_P - d}, \
(\theta, y) \mapsto (\theta, -y)
\end{align*}
canonically extends to a symplectomorphism
\begin{align} \label{eq:0-symplectomorphism}
T^* T^d \times T^* S_P
\hookrightarrow
T^* \widehat{M}_P, \
((\theta, \eta), (x, y)) \mapsto (\theta, -y, \eta, x).
\end{align}
Since it restricts to isomorphisms
\begin{align} \label{eq:0-restriction}
(\tau_P / \langle \sigma_P \rangle)^\perp
\times
\tau_P / \langle \sigma_P \rangle
\times
\sigma_P
\cong
\tau^\perp_P \times \tau_P, \
((\theta, \eta), (x, 0)) \mapsto (\theta, 0, \eta, x),
\end{align}
the symplectomorphism
\pref{eq:0-symplectomorphism}
embeds
$\bL(\Sigma_{S_P}) \times S_P$
into
$\bL(\Sigma_P)$.

(2)
The fanifold
$\Phi_0$
is closed only if
$\dim M_{P, \bR} = 0$
for all $0$-strata
$P$.
Then the claim is trivial.

(3)
Since any subfanifold
$\Phi^\prime_0 \subset \Phi_0$
is the disjoint union of fans
$\Sigma_P$
for some $0$-strata of
$\Phi$,
by definition
$\widetilde{\bfW}(\Phi^\prime_0)
\subset
\widetilde{\bfW}(\Phi_0)$
is a Weinstein submanifold
and
we have
$\widetilde{\bL}(\Phi^\prime_0)
=
\widetilde{\bfW}(\Phi^\prime_0)
\cap
\widetilde{\bL}(\Phi_0)$.

(4)
Recall that
a Lagrangian polarization of a symplectic manifold
$W$
is a global section of the Lagrangian Grassmannian bundle
$\LGr(W)$
over
$W$.
Up to homotopy,
a Lagrangian polarization of
$W$
is equivalent to a real vector bundle
$B$
over
$W$
with an isomorphim
$B \otimes_\bR \bC = TW$.
When
$W$
is a cotangent bundle
$T^* \widehat{M}$,
the tautological foliation by cotangent fibers yields an isomorphism
$T T^* \widehat{M} = T^* \widehat{M} \otimes_\bR \bC$.
Consider the Lagrangian foliation of
$T^* T^d \times T^* S_P$
with leaves 
\begin{align*}
(T^*_\theta T^d \times \{ y \})
\times
(\{ \theta \} \times S_P), \
\theta \in T^d, \
y \in \bR^{\dim \widehat{M}_P - d} \cong S_P.
\end{align*}
Via the symplectomorphism
\pref{eq:0-foliation}
the leaf space get identified with the zero section,
which in turn is isomorphic to
$\pi^{-1}_0(S_P)$.
Along the base direction in
$T^* S_P$,
it is compatible with the inclusion
$S_P \hookrightarrow \overline{S^\prime_P}$
to any stratum
$S^\prime_P$
of codimension
$d^\prime$
of the induced stratification on
$\Nbd(S_P)$.
Along the fiber direction in
$T^* T^d$,
it is compatible with the inclusions
$T^*_\theta T^{d^\prime} \hookrightarrow T^*_\theta T^d$
induced by the quotient map
$T_{S_P} M_{P, \bR} \to T_{S^\prime_P} M_{P, \bR} |_{S_P}$
from Definition
\pref{dfn:fanifold}
for
$\theta \in T^{d^\prime} \cap T^d \subset \widehat{M}_P$.
Hence one obtains the desired polarization.
\end{proof}

\subsection{Special case}
Before moving to general cases,
we explicitly write down
the relevant results
and
their proofs
for Example
\pref{eg:closed},
where the filtration of
$\Phi$
from
\pref{lem:filtration}
is given by
\pref{eq:filters}.
Let
$\cL_1$
be the disjoint union of
\begin{align*}
\cL_{I_{ij}}
=
\pi^{-1}_0(I_{ij}) \cap \del_\infty \widetilde{\bL}(\Phi_0), \
1 \leq i < j \leq 4
\end{align*}
for all $1$-strata
and
\begin{align*}
\cL^1_{F}
=
\pi^{-1}_0(F) \cap \del_\infty \widetilde{\bL}(\Phi_0)
\end{align*}
for the $2$-stratum.

\begin{lemma} \label{lem:SLE1}
There are smooth Legendrian embeddings
\begin{align*}
\del_\infty \widetilde{\bfW}(\Phi_0)
\hookleftarrow
\cL_{I_{ij}},
\cL^1_{F}
\hookrightarrow
\del \bigsqcup_{1 \leq i < j \leq 4} (T^* \widehat{M}_{I_{ij}} \times T^* I_{ij, \circ}). 
\end{align*}
\end{lemma}
\begin{proof}
It suffices to show the claim for
$\cL_{I_{12}}$
and
$\cL^1_{F}$.
Since 
$\del_{in} I_{12}$
consists of two $0$-strata
$P_1, P_2$
of
$\Phi$,
there is a smooth Legendrian embedding
\begin{align} \label{eq:I_12}
\cL_{I_{12}}
\cong
\del_\infty \bL_{\sigma^{P_1}_{I_{12}}}
\sqcup
\del_\infty \bL_{\sigma^{P_2}_{I_{12}}}
\hookrightarrow
\del_\infty \widetilde{\bfW}(\Phi_0)
\end{align}
for the cones
$\sigma^{P_i}_{I_{12}} \in \Sigma_{P_i}$
corresponding to
$I_{12}$.
The quotient maps
$M_{P_i, \bR}
\to
M_{P_i, \bR} / \langle \sigma^{P_i}_{I_{12}} \rangle
=
M_{I_{12}, \bR}$
from Definition
\pref{dfn:fanifold}
identify the images of
$\sigma^{P_i}_{I_{12}}$
with the origin of
$\Sigma_{I_{12}}$.
Hence
$\del_\infty \bL_{\sigma^{P_i}_{I_{12}}}$
are isomorphic to
$\widehat{M}_{I_{12}}$.
Regarding
$\widehat{M}_{I_{12}} \times \del I_{12, \circ}$
as a subset of
$T^* \widehat{M}_{I_{12}} \times \del T^* I_{12, \circ}$,
one obtains another smooth Legendrian embedding
\begin{align*}
\cL_{I_{12}}
\cong 
\widehat{M}_{I_{12}} \times \del I_{12, \circ}
\hookrightarrow
\del (T^* \widehat{M}_{I_{12}} \times T^* I_{12, \circ}).
\end{align*}

Since 
$\cL^1_F$
consists of four disjoint curves in
$F$,
there is a smooth Legendrian embedding
\begin{align} \label{eq:F}
\cL^1_F
\cong
\bigsqcup^4_{i=1} \del_\infty \bL_{\sigma^{P_i}_F}
\hookrightarrow
\del_\infty \widetilde{\bfW}(\Phi_0)
\end{align}
for the cones
$\sigma^{P_i}_F \in \Sigma_{P_i}$
corresponding to
$F$.
The quotient maps
$M_{P_i, \bR}
\to
M_{P_i, \bR} / \langle \sigma^{P_i}_{I_{12}} \rangle
=
M_{I_{ij}, \bR},
1 \leq i, j \leq 4$
from Definition
\pref{dfn:fanifold}
identify the images of
$\sigma^{P_i}_F$
with the rays of
$\Sigma_{I_{ij}}$.
Hence
$\del_\infty \bL_{\sigma^{P_i}_F}$
are isomorphic to the positive half line in the cotangent fiber of
$T^* \widehat{M}_{I_{12}}$
at
$x \in \widehat{M}_{I_{12}}$.
Regarding it as a subset of
$T^* \widehat{M}_{I_{12}} \times \del T^* I_{12, \circ}$,
one obtains another smooth Legendrian embedding
\begin{align*}
\cL^1_F
\hookrightarrow
T^*_x \widehat{M}_{I_{12}} \times \del I_{12, \circ}
\hookrightarrow
\del (T^* \widehat{M}_{I_{12}} \times T^* I_{12, \circ}).
\end{align*}
\end{proof}

\begin{remark}
In the proof of
\pref{thm:fibration},
apparently the authors seem not to consider
$\cL^1_F$
at this stage.
Without
$\cL^1_F$
we would not have the correct handle attachment locus.
\end{remark}

Provided
the chart
\pref{eq:charts}
for
$k = 0$
and
\pref{lem:coordinates}(4),
we may define
$\widetilde{\bfW}(\Phi_1)$
as the handle attachment
\begin{align*}
[\widetilde{\bfW}(\Phi_0)] \#_{\cL_1}  \bigsqcup_{1 \leq i < j \leq 4} [T^* \widehat{M}_{I_{ij}} \times T^* I_{ij, \circ}].
\end{align*}

\begin{lemma} \label{lem:1-W-subanalytic}
The manifold-with-boundary
$\widetilde{\bfW}(\Phi_1)$
is subanalytic Weinstein.
\end{lemma}
\begin{proof}
As a result of the handle attachment,
$\widetilde{\bfW}(\Phi_1)$
is Weinstein.
Each
$T^* \widehat{M}_{I_{ij}} \times T^* I_{ij, \circ}$
and
the union of subanalytic subsets
are subanalytic.
Hence
$\widetilde{\bfW}(\Phi_1)$
is subanalytic by
\pref{lem:0-W-subanalytic}.
\end{proof}

\begin{lemma} \label{lem:1-biconicity}
There is a standard neighborhood
\begin{align} \label{eq:stdNbd1}
\eta_1
\colon
\Nbd_{\del_\infty \widetilde{\bfW}(\Phi_0)}(\cL_1)
\hookrightarrow
J^1 \cL_1
\end{align}
near
$\cL_1$
for which
$\widetilde{\bL}(\Phi_0)$
is biconic along
$\cL_1$
and
$\del_\infty \widetilde{\bL}(\Phi_0)$
locally factors as
\begin{align*}
\eta_1(\widetilde{\bL}(\Phi_0) \cap \Nbd_{\del_\infty \widetilde{\bfW}(\Phi_0)}(\cL_1))
=
\bigsqcup_{1 \leq i < j \leq 4}
(\bL(\Sigma_{I_{ij}}) \times \del I_{ij, \circ}).
\end{align*}
\end{lemma}
\begin{proof}
It suffices to show the claim for
$\cL_{I_{12}}$.
Since we have
\pref{eq:I_12},
the disjoint union
$\bL(\Sigma_{P_1}) \sqcup \bL(\Sigma_{P_2})$
is biconic along
$\cL_{I_{12}}$
for the standard coordinates
\begin{align*}
\eta_{\sigma^{P_1}_{I_{12}}}
\sqcup
\eta_{\sigma^{P_2}_{I_{12}}}
\colon
\Nbd_{\del_\infty T^* \widehat{M}_{P_1} \sqcup \del_\infty T^* \widehat{M}_{P_2}}(\cL_{I_{12}})
\hookrightarrow
J^1 \cL_{I_{12}}
\end{align*}
from
\pref{lem:coordinates}.
Then the local factorization of
$\del_\infty \bL(\Sigma_{P_1})
\sqcup
\del_\infty \bL(\Sigma_{P_2})$
follows from
\pref{eq:recursive}.
\end{proof}

We define
$\widetilde{\bL}(\Phi_1)$
as the extension through the disjoint union of the handles
$T^* \widehat{M}_{I_{ij}} \times T^* I_{ij, \circ}$
\begin{align*}
[\widetilde{\bL}(\Phi_0)] \#_{\cL_1}  \bigsqcup_{1 \leq i < j \leq 4} [\bL(\Sigma_{I_{ij}}) \times I_{ij, \circ}].
\end{align*}

\begin{lemma} \label{lem:1-L-subanalytic}
The Lagrangian
$\widetilde{\bL}(\Phi_1)
\subset
\widetilde{\bfW}(\Phi_1)$
is subanalytic conic
and
contains
$\Core(\widetilde{\bfW}(\Phi_1))$.
\end{lemma}
\begin{proof}
Since by
\pref{lem:1-biconicity}
and
\pref{lem:biconicity-criterion}
the Lagrangians
\begin{align*}
\widetilde{\bL}(\Phi_0) \subset \widetilde{\bfW}(\Phi_0), \
\bigsqcup_{1 \leq i < j \leq 4} [\bL(\Sigma_{I_{ij}}) \times I_{ij, \circ}]
\subset
\bigsqcup_{1 \leq i < j \leq 4} [T^* \widehat{M}_{I_{ij}} \times T^* I_{ij, \circ}]
\end{align*}
are biconic along
$\cL_1$,
the gluing
$\widetilde{\bL}(\Phi_1)$
remains conic.
The part of
$\Core(\widetilde{\bfW}(\Phi_1))$
newly formed by the handle attachment is the saturation of the zero set of the Liouville vector field on
\begin{align*}
\eta_1(\Nbd_{\del_\infty \widetilde{\bfW}(\Phi_0)}(\cL_{I_{ij}}))
\#_{\cL_{I_{ij}}}
\bigsqcup_{1 \leq i < j \leq 4} \Nbd_{\del [T^* \widehat{M}_{I_{ij}} \times T^* I_{ij, \circ}]}(\cL_{I_{ij}})
\subset
[\widetilde{\bfW}(\Phi_0)] \#_{\cL_1}  \bigsqcup_{1 \leq i < j \leq 4} [T^* \widehat{M}_{I_{ij}} \times T^* I_{ij, \circ}],
\end{align*}
which implies
$\Cone(\widetilde{\bfW}(\Phi_1))
\subset
\widetilde{\bL}(\Phi_1)$.
The  extension
$\widetilde{\bL}(\Phi_1)$
is subanalytic,
as it is the union of products of
a real torus
and
an algebraic subset of a Euclidean space
\end{proof}

Let
$\pi_1 \colon \widetilde{\bL}(\Phi_1) \to \Phi_1$
be the map induced by
$\pi_0$
and
the projections from
$T^* \widehat{M}_{I_{ij}} \times T^* I_{ij, \circ}$
to
the cotangent fiber direction in
$T^* \widehat{M}_{I_{ij}}$
and
the base direction in
$T^* I_{ij, \circ}$.

\begin{lemma} \label{lem:1-original}
The triple
$(\widetilde{\bfW}(\Phi_1), \widetilde{\bL}(\Phi_1), \pi_1)$
satisfies the conditions
(1), \ldots, (4).
\end{lemma}
\begin{proof}
(1)
It suffices to show the claim
when attaching the handle
$T^* \widehat{M}_{I_{12}} \times T^* I_{12, \circ}$
to
$\widetilde{\bfW}(\Phi_0)$.
Then we may assume that
\begin{align} \label{eq:1-filtration}
\Phi_0 = \Sigma_{P_1} \sqcup \Sigma_{P_2}, \
\Phi_1 = \Phi_0 \#_{\Sigma_{I_{12}} \times \del_{in} {I_{12}}} (\Sigma_{I_{12}} \times I_{12}).
\end{align}
Let
$S \subset \Phi_1$
be the stratum
\begin{align*}
(S_{P_1} \sqcup S_{P_2}) \#_{\sigma_{P_1} / \langle \sigma^{P_1}_{I_{12}} \rangle \times \del I_{12, \circ}} (\sigma_{P_1} / \langle \sigma^{P_1}_{I_{12}} \rangle \times I_{12, \circ})
\end{align*}
where
$S_{P_i}$
are the strata of the same codimension
$d = 0, 1$
corresponding to cones
$\sigma_{P_i} \in \Sigma_{P_i}$
via
$\phi_{P_i}$.
Then by definition of
$\pi_1$
and
\pref{lem:1-biconicity}
the inverse image
$\pi^{-1}_1(S)$
equals
\begin{align*}
\left( \bigsqcup^2_{i=1}[\bL_{\sigma_{P_i}}] \right)
\#_{\bL_{\sigma_{P_1} / \sigma^{P_1}_{I_{12}}} \times \del I_{12, \circ}}
[\bL_{\sigma_{P_1} / \sigma^{P_1}_{I_{12}}} \times I_{12, \circ}]
\cong
T^d \times \left( \left( \bigsqcup^2_{i=1}S_{P_i} \right) \#_{\sigma_{P_1} / \langle \sigma^{P_1}_{I_{12}} \rangle \times \del I_{12, \circ}} (\sigma_{P_1} / \langle \sigma^{P_1}_{I_{12}} \rangle \times I_{12, \circ}) \right).
\end{align*}

From the isomorphisms
\pref{eq:0-decomposition}
and
\begin{align} \label{eq:1-decomposition}
\tau^\perp_{P_i} \times \tau_{P_i}
\cong
(\tau_{P_i} / \langle \sigma^{P_i}_{I_{12}} \rangle)^\perp
\times
\tau_{P_i} / \langle \sigma_{P_i} \rangle
\times
\sigma_{P_i} / \langle \sigma^{P_i}_{I_{12}} \rangle
\times
I_{12}
\end{align}
for any stratum
$\tau_{P_i} \in \Sigma_{P_i}$
with
$\sigma_{P_i} \subset \bar{\tau}_{P_i}$
it follows
\begin{align*}
\pi^{-1}_1(\Nbd(S))
=
\bigcup_{\tau_{P_1}, \tau_{P_2}}
\left([ \bL_{\tau_{P_1}} \sqcup \bL_{\tau_{P_2}}]
\#_{\bL_{\tau_{P_1} / \sigma^{P_1}_{I_{12}}} \times \del I_{12, \circ}}
[\bL_{\tau_{P_1} / \sigma^{P_1}_{I_{12}}} \times I_{12, \circ}] \right)
\cong
\bL(\Sigma_{P_1} / \sigma_{P_1}) \times S,
\end{align*}
where
$\tau_{P_i}$
run through cones in
$\Sigma_{P_i}$
with
$\sigma_{P_i} \subset \bar{\tau}_{P_i}$
mapping to the same cone under the quotient maps
$M_{P_i, \bR}
\to
M_{P_i, \bR} / \langle \sigma_{P_i} \rangle
=
M_{S, \bR}$
from Definition
\pref{dfn:fanifold}.

Consider the symplectomorphism
\begin{align} \label{eq:1-symplectomorphism}
T^* T^d \times T^* S
\hookrightarrow 
[T^* \widehat{M}_{P_1} \sqcup T^* \widehat{M}_{P_2}]
\#_{\cL_1}
[T^* \widehat{M}_{I_{12}} \times T^* I_{12, \circ}]
\end{align}
induced by
the symplectomorphisms
\begin{align} \label{eq:1-symplectopieces}
T^* T^d \times T^* S_{P_i}
\hookrightarrow
T^* \widehat{M}_{P_i}, \
T^* T^d \times T^* (\sigma_{P_1} / \langle \sigma^{P_1}_{I_{12}} \rangle)
\hookrightarrow
T^* \widehat{M}_{I_{12}}
\end{align}
defined as
\pref{eq:0-symplectomorphism}
and
the identity on
$T^* I_{12, \circ}$.
Since it restricts to isomorphisms
\begin{align} \label{eq:1-restriction}
\bL(\Sigma_{S_{P_i}}) \times S_{P_i}
\hookrightarrow
\bL(\Sigma_{P_i}), \
\bL(\Sigma_{I_{12}} / (\sigma_{P_1} / \langle \sigma^{P_1}_{I_{12}} \rangle)) \times \sigma_{P_1} / \langle \sigma^{P_1}_{I_{12}} \rangle
\hookrightarrow
\bL(\Sigma_{I_{12}})
\end{align}
defined as
\pref{eq:0-restriction},
the symplectomorphism
\pref{eq:1-symplectomorphism}
embeds
$\bL(\Sigma_S) \times S$
into
$\widetilde{\bL}(\Phi_1)$.

(2)
Here
$\Phi_1$
is not closed.

(3)
It suffices to show the claim for
$\Phi^\prime_1 = \Phi_0$.
Then by
definition
and
Remark
\pref{rmk:sector}
\begin{align*}
\widetilde{\bfW}(\Phi^\prime_1)
=
T^* \widehat{M}_{P_1}
\sqcup 
T^* \widehat{M}_{P_2}
\end{align*}
determines a stopped Weinstein sector in
$\widetilde{\bfW}(\Phi_1)$
with relative skeleton
$\widetilde{\bL}(\Phi^\prime_1)
=
\widetilde{\bfW}(\Phi^\prime_1)
\cap
\widetilde{\bL}(\Phi_1)$.
Here,
one obtains
$\widetilde{\bfW}(\Phi^\prime_1)$
by completing
$[T^* \widehat{M}_{P_1}]
\sqcup 
[T^* \widehat{M}_{P_2}]$
along the modified Liouville flow.

(4)
Consider the Lagrangian foliation of
$T^* T^d \times T^* S$
with leaves 
\begin{align*}
(T^*_\theta T^d \times \{ y \})
\times
(\{ \theta \} \times T^*_y S), \
\theta \in T^d, \
y \in S.
\end{align*}
Via the symplectomorphism
$T^* T^d \times T^* S \cong T^* T^d \times T^* S$
induced by
\pref{eq:1-symplectomorphism},
the leaf space get identified with the zero section,
which in turn is isomorphic to
$\pi^{-1}_1(S)$.
For the rest,
the proof of
\pref{lem:0-original}(4)
carries over.
\end{proof}


In our current setting,
$F$
is the only top dimensional stratum of
$\Phi$,
which is interior.
Let
\begin{align*}
\cL_2
=
\cL_ F
=
\pi^{-1}_1(F) \cap \del_\infty \widetilde{\bL}(\Phi_1).
\end{align*}

\begin{lemma}
There are smooth Legendrian embeddings
\begin{align*}
\del_\infty \widetilde{\bfW}(\Phi_1)
\hookleftarrow
\cL_F
\hookrightarrow
\del (T^* \widehat{M}_F \times T^* F_\circ)
\cong
\del (T^* F_\circ). 
\end{align*}
\end{lemma}
\begin{proof}
Since
$\pi^{-1}_1(F)$
is the gluing of
\begin{align*}
\bigsqcup^4_{i=1} \bL_{\sigma^{P_i}_F}, \
\bigsqcup_{1 \leq i < j \leq 4} \bL_{\sigma^{P_i}_F / \sigma^{P_i}_{I_{ij}}} \times I_{ij, \circ}
\end{align*}
by the second bullet of
\pref{lem:1-original}(1),
there is a smooth Legendrian embedding
\begin{align*}
\cL_F
\cong
\widehat{M}_F \times \del F_\circ
\cong
\del F_\circ
\hookrightarrow
\del_\infty \widetilde{\bfW}(\Phi_1).
\end{align*}
On the other hand,
$\widehat{M}_F \times \del F_\circ$
can be regarded as a smooth Legendrian in
\begin{align*}
T^* \del F_\circ
\cong
T^* \widehat{M}_F \times T^* \del F_\circ
\subset
\del (T^* \widehat{M}_F \times T^* F_\circ)
\cong
\del T^* F_\circ.
\end{align*}
\end{proof}

We define
$\widetilde{\bfW}(\Phi_2)$
as the handle attachment
\begin{align*}
[\widetilde{\bfW}(\Phi_1)] \#_{\cL_2} [T^* \widehat{M}_F \times T^* F_\circ]
\cong
[\widetilde{\bfW}(\Phi_1)] \#_{\cL_2} [T^* F_\circ].
\end{align*}

\begin{lemma}
The manifold-with-boundary
$\widetilde{\bfW}(\Phi_2)$
is subanalytic Weinstein.
\end{lemma}
\begin{proof}
The proof of
\pref{lem:1-W-subanalytic}
carries over.
\end{proof}

\begin{lemma} \label{lem:2-biconicity}
There is a standard neighborhood
\begin{align} \label{eq:stdNbd2}
\eta_2
\colon
\Nbd_{\del_\infty \widetilde{\bfW}(\Phi_1)}(\cL_2)
\hookrightarrow
J^1 \cL_2
\end{align}
near
$\cL_2$
for which
$\widetilde{\bL}(\Phi_1)$
is biconic along
$\cL_2$
and
$\del_\infty \widetilde{\bL}(\Phi_1)$
locally factors as
\begin{align*}
\eta_2(\widetilde{\bL}(\Phi_1) \cap \Nbd_{\del_\infty \widetilde{\bfW}(\Phi_1)}(\cL_2))
=
\bL(\Sigma_F) \times \del F_\circ
\cong
\del F_\circ.
\end{align*}
\end{lemma}
\begin{proof}
Along the smooth Legendrian
\begin{align*}
\bigsqcup^4_{i=1} \del_\infty \bL_{\sigma^{P_i}_F}
\subset
\cL_F \cap \del_\infty \widetilde{\bL}(\Phi_0),
\end{align*}
the disjoint union
$\bigsqcup^4_{i=1} \bL(\Sigma_{P_i})$
is biconic for the standard coordinates
\begin{align*}
\bigsqcup^4_{i=1} \eta_{\sigma^{P_i}_F}
\colon
\Nbd_{\del_\infty \widetilde{\bfW}(\Phi_0)}(\cL_F \cap \del_\infty \widetilde{\bL}(\Phi_0))
\hookrightarrow
J^1 \cL_F
\end{align*}
from
\pref{lem:coordinates}.
Along the smooth Legendrian
\begin{align*}
\bigsqcup_{1 \leq i < j \leq 4} (\del_\infty \bL_{\sigma^{P_i}_F / \sigma_{I_{ij}}} \times I_{ij, \circ})
\subset
\bigsqcup_{1 \leq i < j \leq 4} \del_\infty T^* \widehat{M}_{I_{ij}} \times T^* I_{ij, \circ}
\subset
\bigsqcup_{1 \leq i < j \leq 4} \del_\infty (T^* \widehat{M}_{I_{ij}} \times T^* I_{ij, \circ}),
\end{align*}
the Lagrangian
$\bigsqcup_{1 \leq i < j \leq 4} (\bL(\Sigma_{I_{ij}}) \times I_{ij, \circ})$
is biconic for the disjoint union
\begin{align*}
\bigsqcup_{1 \leq i < j \leq 4} (\eta_{\sigma^{P_i}_F / \sigma_{I_{ij}}} \times \can_{I_{ij}})
\colon
\Nbd_{\del_\infty (T^* \widehat{M}_{I_{ij}} \times T^* I_{ij, \circ})}(\cL_F \cap \bigsqcup_{1 \leq i < j \leq 4} \del_\infty (\bL(\Sigma_{I_{ij}}) \times I_{ij, \circ}))
\hookrightarrow
J^1 \cL_F
\end{align*}
of the product of the standard coordinates from
\pref{lem:coordinates}
and
the canonical coordinates
$\can_{I_{ij}}$
on
$\bigsqcup_{1 \leq i < j \leq 4} T^* I_{ij, \circ}$.
By construction of
$\widetilde{\bL}(\Phi_1)$
these coordinates glue to define a standard neighborhood
\begin{align*}
\eta_F
\colon
\Nbd_{\del_\infty \widetilde{\bfW}(\Phi_1)}(\cL_F)
\hookrightarrow
J^1 \cL_F.
\end{align*}
Then the factorization of
$\left( \bigsqcup^4_{i=1} \del_\infty \bL(\Sigma_{P_i}) \right)
\#
\left( \bigsqcup_{1 \leq i < j \leq 4} \del_\infty (\bL(\Sigma_{I_{ij}}) \times I_{ij, \circ}) \right)$
as
\begin{align*}
\left( \bigsqcup^4_{i=1} (\bL(\Sigma_{P_i} / \sigma^{P_i}_F) \times \del_\infty \sigma^{P_i}_F) \right)
\#
\left( \bigsqcup_{1 \leq i < j \leq 4} (\bL(\Sigma_{I_{ij}} / (\sigma^{P_i}_F / \langle \sigma^{P_i}_{I_{ij}} \rangle)) \times \del_\infty \sigma^{P_i}_F / \langle \sigma^{P_i}_{I_{ij}} \rangle \times I_{ij, \circ}) \right)
\cong
\bL(\Sigma_F) \times \del F_\circ
\end{align*}
follows from
\pref{eq:recursive}
for any strata
$\tau_{P_i} \in \Sigma_{P_i}$
with
$\sigma^{P_i}_F \subset \bar{\tau}_{P_i}$
mapping to the same cone in
$\Sigma_F$
under the quotient maps
$\Sigma_{P_i}
\to
\Sigma_{P_i} / \sigma^{P_i}_F$.
\end{proof}

We define
$\widetilde{\bL}(\Phi_2)$
as the extension through the handle
$T^* \widehat{M}_F \times T^* F_\circ
\cong
T^* F_\circ$
\begin{align*}
[\widetilde{\bL}(\Phi_1)] \#_{\cL_2} [\bL(\Sigma_F) \times F_\circ]
\cong
[\widetilde{\bL}(\Phi_1)] \#_{\cL_2} [F_\circ].
\end{align*}

\begin{lemma}
The Lagrangian
$\widetilde{\bL}(\Phi_2) \subset \widetilde{\bfW}(\Phi_2)$
is subanalytic conic
and
contains
$\Core(\widetilde{\bfW}(\Phi_2))$.
\end{lemma}
\begin{proof}
Since by
\pref{lem:2-biconicity}
and
\pref{lem:biconicity-criterion}
the Lagrangians
$\widetilde{\bL}(\Phi_1)
\subset
\widetilde{\bfW}(\Phi_1)$
and
$[F_\circ]
\cong
[\bL(\Sigma_F) \times F_\circ]
\subset
[T^* \widehat{M}_F \times T^* F_\circ]
\cong
[T^* F_\circ]$
are biconic along
$\cL_2$,
the gluing
$\widetilde{\bL}(\Phi_2)$
remains conic.
For the rest of the claim,
the proof of
\pref{lem:1-L-subanalytic}
carries over.
\end{proof}

Let
$\pi_2 \colon \widetilde{\bL}(\Phi_2) \to \Phi_2$
be the map induced by
$\pi_1$
and
the projection
$T^* F_\circ \to F_\circ$
to the base.

\begin{lemma} \label{lem:2-original}
The triple
$(\widetilde{\bfW}(\Phi_2), \widetilde{\bL}(\Phi_2), \pi_2)$
satisfies the conditions
(1), \ldots, (4).
\end{lemma}
\begin{proof}
(1)
The only top dimensional stratum
$F$
can be written as
\begin{align*}
\left( \left( \bigsqcup^4_{i=1} S_{P_i} \right)
\#_{\bigsqcup_{1 \leq i < j \leq 4} (\sigma_{P_i} / \langle \sigma^{P_i}_{I_{ij}} \rangle \times \del I_{ij, \circ})}
\left( \bigsqcup_{1 \leq i < j \leq 4} \sigma_{P_i} / \langle \sigma^{P_i}_{I_{ij}} \rangle \times I_{ij, \circ} \right) \right)
\#_{\sigma_{P_1} / \langle \sigma^{P_1}_F \rangle \times \del F_\circ}
(\sigma_{P_1} / \langle \sigma^{P_1}_F \rangle \times F_\circ),
\end{align*}
where
$S_{P_i} \subset \Phi_0, 1 \leq i \leq 4$
are the strata of codimension
$0$
corresponding to cones
$\sigma_{P_i} \in \Sigma_{P_i}$
via
$\phi_{P_i}$.
Then by definition of
$\pi_2$
and
\pref{lem:2-biconicity}
the inverse image
$\pi^{-1}_2(F)$
equals
\begin{align*}
\left( [\bigsqcup^4_{i=1} \bL_{\sigma_{P_i}}]
\#_{\bigsqcup_{1 \leq i < j \leq 4} ( \bL_{\sigma_{P_i} / \sigma^{P_i}_{I_{ij}}} \times \del I_{ij, \circ} )}
[\bigsqcup_{1 \leq i < j \leq 4} \bL_{\sigma_{P_i} / \sigma^{P_i}_{I_{ij}}} \times I_{ij, \circ}] \right)
\#_{\bL_{\sigma_{P_1} / \sigma^{P_1}_F} \times \del F_\circ}
[\bL_{\sigma_{P_1} / \sigma^{P_1}_F} \times F_\circ]
\cong
F.
\end{align*}
Since
$F$
is top dimensional,
$\pi^{-1}_2(\Nbd(F))$
coincides with
$\pi^{-1}_2(F) \cong F \cong \bL(\Sigma_F) \times F$.
For the other strata
$S \subset \Phi_2$
of positive codimension,
the claims are obvious
as we have
$\pi^{-1}_2(S)
=
\pi^{-1}_1(S)$.

Consider the symplectomorphism
\begin{align} \label{eq:2-symplectomorphism}
T^* T^d \times T^* S
\hookrightarrow 
\left( [\bigsqcup^4_{i=1} T^* \widehat{M}_{P_i}]
\#_{\bigsqcup_{1 \leq i < j \leq 4} \cL_{I_{ij}}}
[\bigsqcup_{1 \leq i < j \leq 4} T^* \widehat{M}_{I_{ij}} \times T^* {I_{ij, \circ}}] \right)
\#_{\cL_F}
[T^* F_\circ]
\end{align}
induced by
the symplectomorphisms
\pref{eq:1-symplectopieces}
and
the identity on
$T^* I_\circ, T^* F_\circ$.
Since it restricts to
such isomorphisms as
\pref{eq:1-restriction}
and
$\Sigma_F$
is trivial,
the symplectomorphism
\pref{eq:2-symplectomorphism}
embeds
$\bL(\Sigma_S) \times S$
into
$\widetilde{\bL}(\Phi_2)$.

(2)
Here
$\Phi_2$
is closed.
The saturation of the zero set of the Liouville vector field on
\begin{align*}
\eta_2(\Nbd_{\del_\infty \widetilde{\bfW}(\Phi_1)}(\cL_2))
\#_{\cL_2}
\Nbd_{\del [T^* F_\circ]}(\cL_F)
\subset
[\widetilde{\bfW}(\Phi_1)] \#_{\cL_2} [T^* F_\circ]
\end{align*}
gives the newly formed part of
$\Core(\widetilde{\bfW}(\Phi_2))$
by the handle attachment.
Due to the absence of higher dimensional strata,
it projects onto
$F$
under
$\pi_2$.
Since
$F$ is interior,
by
\pref{lem:1-original}(2)
its union with
$\widetilde{\bL}(\Phi_1)$
coincides with
$\widetilde{\bL}(\Phi_2)$.

(3)
It suffices to show the claim for
$\Phi^\prime_2 = \Phi_1$.
Then by
definition
and
Remark
\pref{rmk:sector}
\begin{align*}
\widetilde{\bfW}(\Phi^\prime_2)
=
[\bigsqcup^4_{i=1} T^* \widehat{M}_{P_i}]
\#_{\bigsqcup_{1 \leq i < j \leq 4} \cL_{I_{ij}}}
[\bigsqcup_{1 \leq i < j \leq 4} T^* \widehat{M}_{I_{ij}} \times T^* I_{ij, \circ}]
\end{align*}
determines a stopped Weinstein sector in
$\widetilde{\bfW}(\Phi_2)$
with relative skeleton
$\widetilde{\bL}(\Phi^\prime_2)
=
\widetilde{\bfW}(\Phi^\prime_2) \cap \widetilde{\bL}(\Phi_2)$.
Here,
one obtains
$\widetilde{\bfW}(\Phi^\prime_2)$
by completing the gluing of
$[\bigsqcup^4_{i=1} T^* \widehat{M}_{P_i}]$
with
$[\bigsqcup_{1 \leq i < j \leq 4} T^* \widehat{M}_{I_{ij}} \times T^* I_{ij, \circ}]$
along the modified Liouville flow.

(4)
Consider the Lagrangian foliation of
$T^* T^d \times T^* S$
with leaves 
\begin{align*}
(T^*_\theta T^d \times \{ y \})
\times
(\{ \theta \} \times T^*_y S), \
\theta \in T^d, \
y \in S.
\end{align*}
Via the symplectomorphism
$T^* T^d \times T^* S \cong T^* T^d \times T^* S$
induced by
\pref{eq:2-symplectomorphism},
the leaf space get identified with the zero section,
which in turn is isomorphic to
$\pi^{-1}_2(S)$.
For the rest,
the proof of
\pref{lem:0-original}(4)
carries over.
\end{proof}



\subsection{General case}
Suppose that
\pref{thm:fibration}
and
the relevant results hold for the subfanifold
$\Phi_{k-1}$.
Let
$\cL_k$
be the disjoint union of
\begin{align*}
\cL^k_ {S^{(l)}}
=
\pi^{-1}_{k-1}(S^{(l)}) \cap \del_\infty \widetilde{\bL}(\Phi_{k-1})
\end{align*}
for all interior $l$-strata
$S^{(l)} \subset \Phi, k \leq l$.

\begin{lemma}
There are smooth Legendrian embeddings
\begin{align*}
\del_\infty \widetilde{\bfW}(\Phi_{k-1})
\hookleftarrow
\cL^k_ {S^{(l)}}
\hookrightarrow
\del \bigsqcup_{S^{(k)}} (T^* \widehat{M}_{S^{(k)}} \times T^* S^{(k)}_\circ). 
\end{align*}
\end{lemma}
\begin{proof}
Here,
we show the claim only for
$\cL_{S^{(k)}}$.
The other cases can be discussed in a similar way to the proof of
\pref{lem:SLE1}.
Since
$\pi^{-1}_{k-1}(S^{(k)})$
is the gluing of
\begin{align*} 
\bL_{\sigma^P_{S^{(k)}}}, \
\bL_{\sigma^P_{S^{(k)}} / \sigma^P_I} \times I_\circ, \
\bL_{\sigma^P_{S^{(k)}} / \sigma^P_F} \times F_\circ, \
\ldots
\end{align*}
by the second bullet of
\pref{thm:fibration}(1)
for
$n = k-1$,
there is a smooth Legendrian embedding
\begin{align*}
\cL_{S^{(k)}}
\cong
\widehat{M}_{S^{(k)}} \times \del S^{(k)}_\circ
\hookrightarrow
\del_\infty \widetilde{\bfW}(\Phi_{k-1}).
\end{align*}
Here,
$P, I, F, \ldots$
run through strata of dimension
$0, 1, 2, \ldots$
in
$\del S^{(k)}_\circ$
with
$P \subset \del_{in} I,
P, I \subset \del_{in} F,
\ldots$.
Note that
by Definition
\pref{dfn:fanifold}
$\sigma^P_{S^{(k)}} / \langle \sigma^P_I \rangle,
\sigma^P_{S^{(k)}} / \langle \sigma^P_F \rangle,
\ldots$
regarded as cones in
$\Sigma_I, \Sigma_F, \ldots$
do not depend on the choice of
$P$
and
we fix some
$P$
when such
$I, F, \ldots$
run. 
On the other hand,
$\widehat{M}_{S^{(k)}} \times \del S^{(k)}_\circ$
can be regarded as a smooth Legendrian in
\begin{align*}
T^* \widehat{M}_{S^{(k)}} \times T^* \del S^{(k)}_\circ
\subset
\del \bigsqcup_{S^{(k)}} (T^* \widehat{M}_{S^{(k)}} \times T^* S^{(k)}_\circ).
\end{align*}
\end{proof}

Provided
the chart
\pref{eq:charts}
for the previous step
and
\pref{lem:coordinates}(4),
we may define
$\widetilde{\bfW}(\Phi_k)$
as the handle attachment
\begin{align*}
[\widetilde{\bfW}(\Phi_{k-1})] \#_{\cL_k}  \bigsqcup_{S^{(k)}} [T^* \widehat{M}_{S^{(k)}} \times T^* S^{(k)}_\circ].
\end{align*}

\begin{lemma}
The manifold-with-boundary
$\widetilde{\bfW}(\Phi_k)$
is subanalytic Weinstein.
\end{lemma}
\begin{proof}
The proof of
\pref{lem:1-W-subanalytic}
carries over.
\end{proof}

\begin{lemma} \label{lem:k-biconicity}
There is a standard neighborhood
\begin{align} \label{eq:stdNbdk}
\eta_k
\colon
\Nbd_{\del_\infty \widetilde{\bfW}(\Phi_{k-1})}(\cL_k)
\hookrightarrow
J^1 \cL_k
\end{align}
near
$\cL_k$
for which
$\widetilde{\bL}(\Phi_{k-1})$
is biconic along
$\cL_k$
and
$\del_\infty \widetilde{\bL}(\Phi_{k-1})$
locally factors as
\begin{align*}
\eta_k(\widetilde{\bL}(\Phi_{k-1})
\cap
\Nbd_{\del_\infty \widetilde{\bfW}(\Phi_{k-1})}(\cL_k))
=
\bigsqcup_{S^{(k)}}
(\bL(\Sigma_{S^{(k)}}) \times \del S^{(k)}_\circ).
\end{align*}
\end{lemma}
\begin{proof}
Here,
we show the claim only for
$\cL_{S^{(k)}}$
to obtain a standard neighborhood
\begin{align*}
\eta_{S^{(k)}}
\colon
\Nbd_{\del_\infty \widetilde{\bfW}(\Phi_{k-1})}(\cL_{S^{(k)}})
\hookrightarrow
J^1 \cL_{S^{(k)}}.
\end{align*}
Due to
\pref{lem:coordinates}(4),
such standard neighborhoods for all
$S^{(l)}, k \leq l$
glue to yield
$\eta_k$.
Along
\begin{align*}
\del_\infty \bL_{\sigma^P_{S^{(k)}}}, \
\del_\infty \bL_{\sigma^P_{S^{(k)}} / \sigma^P_I}  \times I_\circ, \
\del_\infty \bL_{\sigma^P_{S^{(k)}} / \sigma^P_F}  \times F_\circ, \
\ldots
\end{align*}
where
$P, I, F, \ldots$
run through strata of dimension
$0, 1, 2, \ldots$
in
$\del S_\circ$
with
$P \subset \del_{in} I,
P, I \subset \del_{in} F,
\ldots$
as above,
\begin{align} \label{eq:k-Lagrangians}
\bL(\Sigma_P), \
\bL(\Sigma_I) \times I_\circ, \
\bL(\Sigma_F) \times F_\circ, \
\ldots
\end{align}
are biconic for the products
\begin{align*}
\eta_{\sigma^P_{S^{(k)}}}, \
\eta_{\sigma^P_{S^{(k)}} / \sigma^P_I} \times \can_I, \
\eta_{\sigma^P_{S^{(k)}} / \sigma^P_F} \times \can_F, \
\ldots
\end{align*}
of the standard coordinates from
\pref{lem:coordinates}
and
the canonical coordinates
$\can_I, \can_F, \ldots$
on
$T^* I_\circ, T^* F_\circ, \ldots$.
By construction of
$\widetilde{\bL}(\Phi_{k-1})$
these coordinates glue to define
$\eta_{S^{(k)}}$.
Then the factorization of the gluing of
\pref{eq:k-Lagrangians}
as the gluing of
\begin{align*}
\bL(\Sigma_P / \sigma^P_{S^{(k)}}), \
\bL(\Sigma_I / (\sigma^P_{S^{(k)}} / \langle \sigma^P_I \rangle)) \times \sigma^P_{S^{(k)}} / \langle \sigma^P_I \rangle \times I_\circ, \
\bL(\Sigma_F / (\sigma^P_{S^{(k)}} / \langle \sigma^P_F \rangle)) \times \sigma^P_{S^{(k)}} / \langle \sigma^P_F \rangle \times F_\circ, \
\ldots
\end{align*}
follows from
\pref{eq:recursive}
for any strata
$\tau_P \in \Sigma_P,
\tau_P / \langle \sigma^P_I \rangle \in \Sigma_I,
\tau_P / \langle \sigma^P_F \rangle \in \Sigma_F,
\ldots$
with
$\sigma^P_{S^{(k)}} \subset \bar{\tau}_P$
mapping to the same cone in
$\Sigma_{S^{(k)}}$
under the quotient maps
\begin{align*}
\Sigma_P
\to
\Sigma_P / \sigma^P_{S^{(k)}}, \
\Sigma_I
\to
\Sigma_I / (\sigma^P_{S^{(k)}} / \langle \sigma^P_I \rangle), \
\Sigma_F
\to
\Sigma_F / (\sigma^P_{S^{(k)}} / \langle \sigma^P_F \rangle), \
\ldots.
\end{align*}
\end{proof}

We define
$\widetilde{\bL}(\Phi_k)$
as the extension through the disjoint union of the handles
$T^* \widehat{M}_{S^{(k)}} \times T^* S^{(k)}_\circ$
\begin{align*}
[\widetilde{\bL}(\Phi_{k-1})] \#_{\cL_k}  \bigsqcup_{S^{(k)}} [\bL(\Sigma_{S^{(k)}}) \times S^{(k)}_\circ].
\end{align*}

\begin{lemma}
The Lagrangian
$\widetilde{\bL}(\Phi_k) \subset \widetilde{\bfW}(\Phi_k)$
is subanalytic conic
and
contains
$\Core(\widetilde{\bfW}(\Phi_k))$.
\end{lemma}
\begin{proof}
Since by
\pref{lem:k-biconicity}
and
\pref{lem:biconicity-criterion}
the Lagrangians
\begin{align*}
\widetilde{\bL}(\Phi_{k-1}) \subset \widetilde{\bfW}(\Phi_{k-1}), \
\bigsqcup_{S^{(k)}} [\bL(\Sigma_{S^{(k)}}) \times S^{(k)}_\circ]
\subset
\bigsqcup_{S^{(k)}} [T^* \widehat{M}_{S^{(k)}} \times T^* S^{(k)}_\circ]
\end{align*}
are biconic along
$\cL_k$,
the gluing
$\widetilde{\bL}(\Phi_k)$
remains conic.
For the rest of the claim,
the proof of
\pref{lem:1-L-subanalytic}
carries over.
\end{proof}

Let
$\pi_k \colon \widetilde{\bL}(\Phi_k) \to \Phi_k$
be the map induced by
$\pi_{k-1}$
and
the projections from
$T^* \widehat{M}_{S^{(k)}} \times T^* S^{(k)}_\circ$
to
the cotangent fiber direction in
$T^* \widehat{M}_{S^{(k)}}$
and
the base direction in
$T^* S^{(k)}_\circ$.

\begin{lemma} \label{lem:k-original}
The triple
$(\widetilde{\bfW}(\Phi_k), \widetilde{\bL}(\Phi_k), \pi_k)$
satisfies the conditions
(1), \ldots, (4).
\end{lemma}
\begin{proof}
(1)
It suffices to show the claim
when attaching the handle
$T^* \widehat{M}_{S^{(k)}} \times T^* S^{(k)}_\circ$
to
$\widetilde{\bfW}(\Phi_{k-1})$
for a single interior $k$-stratum
$S^{(k)} \subset \Phi$
with
$k \leq \dim \Phi$.
Then by definition of
$\pi_k$
and
\pref{lem:k-biconicity}
the inverse image
$\pi^{-1}_k(S)$
is the gluing of
\begin{align*} 
\bL_{\sigma_P} \cong T^d \times \sigma_P, \
\bL_{\sigma_P / \sigma^P_I} \times I_\circ
\cong
T^d \times \sigma_P / \langle \sigma^P_I \rangle \times I_\circ, \
\bL_{\sigma_P / \sigma^P_F} \times F_\circ
\cong
T^d \times \sigma_P / \langle \sigma^P_F \rangle \times F_\circ, \
\ldots
\end{align*}
and
$\bL_{\sigma^P_{S^{(k)}}} \times S^{(k)}_\circ
\cong
T^d \times \sigma_P / \langle \sigma^P_{S^{(k)}} \rangle \times S^{(k)}_\circ$
where
$P, I, F, \ldots$
run through strata of dimension
$0, 1, 2, \ldots$
in
$\del S^{(k)}_\circ$
with
$P \subset \del_{in} I,
P, I \subset \del_{in} F,
\ldots$
as above.
Hence by
\pref{lem:filtration}
we obtain
$\pi^{-1}_k(S) \cong T^d \times S$.

From the isomorphisms
\pref{eq:0-decomposition},
\pref{eq:1-decomposition},
\ldots
and
\begin{align} \label{eq:k-decomposition}
\tau^\perp_P \times \tau_P
\cong
(\tau_P / \langle \sigma_P \rangle)^\perp
\times
\tau_P / \langle \sigma_P \rangle
\times
\sigma_P / \langle \sigma^P_{S^{(k)}} \rangle
\times
S^{(k)}
\end{align}
for any stratum
$\tau_P \in \Sigma_P$
with
$\sigma_P \subset \bar{\tau}_P$,
it follows that
$\pi^{-1}_k(S^\prime)$
for a stratum
$S^\prime$
of the induced stratification on
$\Nbd(S)$
is the gluing of
\begin{align*} 
\bL_{\tau_P}
\cong
\bL_{\tau_P / \sigma_P} \times \sigma_P, \
\bL_{\tau_P / \sigma^P_I} \times I_\circ
\cong
\bL_{\tau_P / \sigma_P} \times \sigma_P / \langle \sigma^P_I \rangle \times I_\circ, \
\bL_{\tau_P / \sigma^P_F} \times F_\circ
\cong
\bL_{\tau_P / \sigma_P} \times \sigma_P / \langle \sigma^P_F \rangle \times F_\circ, \
\ldots
\end{align*}
and
$\bL_{\tau_P / \sigma_P} \times \sigma_P / \langle \sigma^P_{S^{(k)}} \rangle \times S^{(k)}_\circ$
where
$\tau_P$
are the cones in
$\Sigma_P$
corresponding to
$S^\prime$.
Since
$\tau_P$
map to the same cone under the quotient maps
$M_{P, \bR}
\to
M_{P, \bR} / \langle \sigma_{P} \rangle
=
M_{S, \bR}$
from Definition
\pref{dfn:fanifold},
by
\pref{lem:filtration}
we obtain
\begin{align*}
\pi^{-1}_k(S^\prime) \cong T^{d^\prime} \times \tau_P / \langle \sigma_P \rangle \times S, \
\pi^{-1}_k(\Nbd(S)) \cong \bL(\Sigma_P / \sigma_P) \times S.
\end{align*}

Consider the symplectomorphism
\begin{align} \label{eq:k-symplectomorphism}
T^* T^d \times T^* S
\hookrightarrow 
[\widetilde{\bfW}(\Phi_{k-1})]
\#_{\cL_{S^{(k)}}}
[T^* \widehat{M}_{S^{(k)}} \times T^* S^{(k)}_\circ]
\end{align}
induced
by the symplectomorphisms
\begin{align*}
T^* T^d \times T^* S_P
\hookrightarrow
T^* \widehat{M}_P, \
T^* T^d \times T^* (\sigma_P / \langle \sigma^P_I \rangle)
\hookrightarrow
T^* \widehat{M}_I, \
T^* T^d \times T^* (\sigma_P / \langle \sigma^P_F \rangle)
\hookrightarrow
T^* \widehat{M}_F, \
\ldots
\end{align*}
and the symplectomorphism
\begin{align*}
T^* T^d \times T^* (\sigma_P / \langle \sigma^P_{S^{(k)}} \rangle)
\hookrightarrow
T^* \widehat{M}_{S^{(k)}}
\end{align*}
defined as
\pref{eq:0-symplectomorphism},
and
the identity on
$T^* I_\circ, T^* F_\circ, \ldots$
and
$T^* S^{(k)}_\circ$.
Since it restricts to isomorphisms
\begin{align*}
\bL(\Sigma_{S_P}) \times S_P
\hookrightarrow
\bL(\Sigma_P), \
\bL(\Sigma_I / (\sigma_P / \langle \sigma^P_I \rangle)) \times \sigma_P / \langle \sigma^P_I \rangle
\hookrightarrow
\bL(\Sigma_I), \
\ldots
\end{align*}
and
an isomorphism
\begin{align} \label{eq:k-restriction}
\bL(\Sigma_{S^{(k)}} / (\sigma_P / \langle \sigma^P_{S^{(k)}} \rangle)) \times \sigma_P / \langle \sigma^P_{S^{(k)}} \rangle
\hookrightarrow
\bL(\Sigma_{S^{(k)}})
\end{align}
defined as
\pref{eq:0-restriction},
the symplectomorphism
\pref{eq:k-symplectomorphism}
embeds
$\bL(\Sigma_S) \times S$
into
$\widetilde{\bL}(\Phi_k)$.

(2)
The fanifold
$\Phi_k$
is closed only if
there are no strata of dimension more than
$k$
and
all strata are interior.
Then the saturation of the zero set of the Liouville vector field on
\begin{align*}
\eta_k(\Nbd_{\del_\infty \widetilde{\bfW}(\Phi_{k-1})}(\cL_k))
\#_{\cL_k}
\bigsqcup_{S^{(k)}} \Nbd_{\del [T^* \widehat{M}_{S^{(k)}} \times T^* S^{(k)}_\circ]}(\cL_{S^{(k)}})
\subset
[\widetilde{\bfW}(\Phi_{k-1})] \#_{\cL_{S^{(k)}}}  \bigsqcup_{S^{(k)}} [T^* \widehat{M}_{S^{(k)}} \times T^* S^{(k)}_\circ]
\end{align*}
gives the newly formed part of
$\Core(\widetilde{\bfW}(\Phi_k))$
by the handle attachment.
Due to the absence of higher dimensional strata,
it projects onto
$\bigsqcup_{S^{(k)}} S^{(k)} \subset \Phi_k$
under
$\pi_k$.
Since all $k$-strata are interior,
by
\pref{thm:fibration}(2)
for
$n = k-1$
its union with
$\widetilde{\bL}(\Phi_{k-1})$
coincides with
$\widetilde{\bL}(\Phi_k)$.

(3)
It suffices to show the claim for
\begin{align*}
\Phi_k = \Phi_{k-1} \#_{\Sigma_{S^{(k)}} \times \del_{in} S^{(k)}} (\Sigma_{S^{(k)}} \times S^{(k)}), \
\Phi^\prime_k = \Phi_{k-1}.
\end{align*}
Then by
definition
and
Remark
\pref{rmk:sector}
$\widetilde{\bfW}(\Phi^\prime_k)$
determines a stopped Weinstein sector in
$\widetilde{\bfW}(\Phi_k)$
with relative skeleton
$\widetilde{\bL}(\Phi^\prime_k)
=
\widetilde{\bfW}(\Phi^\prime_k) \cap \widetilde{\bL}(\Phi_k)$.
Here,
one obtains
$\widetilde{\bfW}(\Phi^\prime_k)$
by completing the gluing of the domains along the modified Liouville flow.

(4)
Consider the Lagrangian foliation of
$T^* T^d \times T^* S$
with leaves 
\begin{align*}
(T^*_\theta T^d \times \{ y \})
\times
(\{ \theta \} \times T^*_y S), \
\theta \in T^d, \
y \in S.
\end{align*}
Via the symplectomorphism
$T^* T^d \times T^* S \cong T^* T^d \times T^* S$
induced by
\pref{eq:k-symplectomorphism},
the leaf space get identified with the zero section,
which in turn is isomorphic to
$\pi^{-1}_k(S)$.
Along the base direction in
$T^* S$,
it is compatible with the inclusion
$S \hookrightarrow \overline{S^\prime}$
to any stratum
$S^\prime$
of codimension
$d^\prime$
of the induced stratification on
$\Nbd(S)$.
Along the fiber direction in
$T^* T^d$,
it is compatible with the inclusions
$T^*_\theta T^{d^\prime} \hookrightarrow T^*_\theta T^d$
induced by the quotient maps
\begin{align*}
T_{S_P} M_{P, \bR}
\to
T_{S^\prime_P} M_{P, \bR} |_{S_P}, \
T_{\sigma_P / \langle \sigma^P_I \rangle} M_{I, \bR}
\to
T_{\sigma^\prime_P / \langle \sigma^P_I \rangle} M_{I, \bR} |_{\sigma_P / \langle \sigma^P_I \rangle}, \
\ldots
\end{align*}
and
$T_{\sigma_P / \langle \sigma^P_{S^{(k)}} \rangle} M_{S^{(k)}, \bR}
\to
T_{\sigma^\prime_P / \langle \sigma^P_{S^{(k)}} \rangle} M_{S^{(k)}, \bR} |_{\sigma_P / \langle \sigma^P_{S^{(k)}} \rangle}$
from Definition
\pref{dfn:fanifold}
for
$\theta \in T^{d^\prime} \cap T^d \subset \widehat{M}_P$.
Hence one obtains the desired polarization.
\end{proof}

\begin{remark}
The symplectomorphism
\pref{eq:k-symplectomorphism}
sends the cotangent fibers of
\begin{align*}
T^* S_P, \
T^* (\sigma_P / \langle \sigma^P_I \rangle), \
T^* (\sigma_P / \langle \sigma^P_F \rangle), \
\ldots
\end{align*}
and
$T^* (\sigma_P / \langle \sigma^P_{S^{(k)}} \rangle)$
to the bases of
$T^* \widehat{M}_P, \
T^* \widehat{M}_I, \
T^* \widehat{M}_F, \
\ldots$
and
$T^* \widehat{M}_{S^{(k)}}$
with negation.
\end{remark}

\begin{remark}
If
$\Sigma_P, \Sigma_I, \Sigma_F, \ldots$
and
$\Sigma_{S^{(k)}}$
are stacky fans,
then we consider stacky FLTZ skeleta 
$\bL(\Sigma_{P_i}), \bL(\Sigma_{I_{ij}}), \bL(\Sigma_F), \ldots$
and
$\bL(\Sigma_{S^{(k)}})$.
According to how many of copies of tori there,
duplicate the corresponding handles
$T^* \widehat{M}_I \times T^* I_\circ, \
T^* \widehat{M}_F \times T^* F_\circ, \
\ldots$
and
$T^* \widehat{M}_{S^{(k)}} \times T^* S^{(k)}_\circ$.
Then our proof generalizes in a straightforward way.
\end{remark}

\section{Proof of \pref{thm:lift}}
Recall that
a
\emph{fibration}
is a map
which satisfies the homotopy lifting property for all topological spaces.
Any fiber bundle over a paracompact Hausdorff base gives an example.
In this section,
we first construct an intermediate filtered stratified fibration
$\tilde{\pi}$
from
$\widetilde{\bfW}(\Phi)$
restricting to
$\pi$,
which defines a filtered stratified integrable system with noncompact fibers.
The composition with a certain map induced by retractions yields
$\bar{\pi}$.
When
$\Sigma_S$
is proper for all
$S \subset \Phi$,
the map is trivial
and
$\bar{\pi}$
defines the integrable system.
As in the previous section,
we proceed by induction on
$k$.

Below,
we will use the disjoint union
\begin{align*}
\ret_k
\colon
\bigsqcup_{S^{(k)}} M_{{S^{(k)}}, \bR}
\to
\bigsqcup_{S^{(k)}} \Sigma_{S^{(k)}}
\end{align*}
for all $k$-strata
$S^{(k)} \subset \Phi$
of the maps
$\ret_{S^{(k)}}
\colon
M_{S^{(k)}, \bR}
\to
\Sigma_{S^{(k)}}$
defined as follows.
If
$\Sigma_{S^{(k)}}$
is proper,
then 
$\ret_{S^{(k)}}$
is induced by the identity.
Otherwise,
beginning with higher dimensional strata in
$\del \Sigma_{S^{(k)}}$,
project the half fibers of
$T_{\del \Sigma_{S^{(k)}}} M_{S^{(k)}, \bR}$
which do not intersect
$\Int(\Sigma_{S^{(k)}})$
to each stratum.
In particular,
the last projection will be contraction of
\begin{align*}
M_{S^{(k)}, \bR} \setminus (T_{\del \Sigma_{S^{(k)}} \setminus 0} M_{S^{(k)}, \bR} \cup \Int(\Sigma_{S^{(k)}}))
\end{align*}
to the origin.
The result gives
$\ret_{S^{(k)}}$.

\begin{example}
Consider the map
$\ret_{P_2} \colon M_{P_2, \bR} \to \Sigma_{P_2}$
for the $0$-stratum
$P_2$
in Example \pref{eg:closed}.
The boundary
$\del \Sigma_{P_2}$
is the union
$\{ 0 \} \times [0, \infty) \cup [0, \infty) \times \{ 0 \}$.
Then we have
\begin{align*}
\begin{gathered}
T_{\sigma^{P_2}_{I_{12}} \cup \sigma^{P_2}_{I_{23}}} M_{P_2, \bR}
\cap
(M_{P_2, \bR} \setminus \Int(\Sigma_{P_2}))
=
(-\infty, 0] \times (0, \infty)
\cup
(0, \infty) \times (-\infty, 0], \\
M_{P_2, \bR} \setminus (T_{\del \Sigma_{P_2} \setminus 0} M_{P_2, \bR} \cup \Int(\Sigma_{P_2}))
=
(-\infty, 0] \times (-\infty, 0]
\end{gathered} 
\end{align*}
which
$\ret_{P_2}$
projects to
$\{ 0 \} \times (0, \infty) \cup (0, \infty) \times \{ 0 \}$
and
$\{ 0 \} \times \{ 0 \}$
respectively.
\end{example}

\subsection{Base case}
\begin{lemma} \label{lem:0-lift}
There is a stratified fibration
$\bar{\pi}_0 \colon \widetilde{\bfW}(\Phi_0) \to \Phi_0$
restricting to
$\pi_0$.
\end{lemma}
\begin{proof}
Define
$\tilde{\pi}_0$
as the disjoint union of the projections to the cotangent fibers
$T^* \widehat{M}_P
\to
M_{P, \bR}$.
Clearly,
its restriction to
$\widetilde{\bL}(\Phi_0)$
coincides with
$\pi_0$.
The composition
$\bar{\pi}_0 = \ret_0 \circ \tilde{\pi}_0$
gives the desired fibration.
\end{proof}


\subsection{General case}
Suppose that
\pref{thm:lift}
holds for the subfanifold
$\Phi_{k-1}$.

\begin{lemma} \label{lem:k-lift}
There is a filtered stratified fibration
$\bar{\pi}_k \colon \widetilde{\bfW}(\Phi_k) \to \Phi_k$
restricting to
$\pi_k$.
\end{lemma}
\begin{proof}
It suffices to show the claim
when attaching the handle
$T^* \widehat{M}_{S^{(k)}} \times T^* S^{(k)}_\circ$
to
$\widetilde{\bfW}(\Phi_{k-1})$
for a single interior $k$-stratum
$S^{(k)} \subset \Phi$
with
$k \leq \dim \Phi$.
Define
$\tilde{\pi}_k$
as the map canonically induced by
$\tilde{\pi}_{k-1}, \tilde{\pi}_{0, S^{(k)}}$
and
the projection
$T^* S^{(k)}_\circ \to S^{(k)}_\circ$
to the base,
where
$\tilde{\pi}_{0, S^{(k)}}
\colon
T^* \widehat{M}_{S^{(k)}} \to M_{S^{(k)}, \bR}$
is the projection to the cotangent fibers.
Here,
we precompose the contraction
\begin{align*}
\cont_k \colon \widetilde{\bfW}(\Phi_k) \to \widetilde{\bfW}(\Phi_k)
\end{align*}
of the cylindrical ends along the negative Liouville flow to the complement of the gluing region in
$\del_\infty \widetilde{\bfW}(\Phi_{k-1})$.
Define
$\bar{\pi}_k$
as the map canonically induced by
$\bar{\pi}_{k-1},
\bar{\pi}_{0, S^{(k)}} = \ret_k \circ \tilde{\pi}_{0, S^{(k)}}$
and
the projection
$T^* S^{(k)}_\circ \to S^{(k)}_\circ$
to the base.

Recall that the symplectomorphism
\pref{eq:k-symplectomorphism}
sends the bases of
\begin{align*}
T^* S_P, \
T^* (\sigma_P / \langle \sigma^P_I \rangle), \
T^* (\sigma_P / \langle \sigma^P_F \rangle), \
\ldots
\end{align*}
and
$T^* (\sigma_P / \langle \sigma^P_{S^{(k)}} \rangle)$
to the cotangent fibers of
\begin{align*}
T^* \widehat{M}_P, \
T^* \widehat{M}_I, \
T^* \widehat{M}_F, \
\ldots
\end{align*}
and
$T^* \widehat{M}_{S^{(k)}}$
preserving
$T^* I_\circ, T^* F_\circ, \ldots$,
where
$P, I, F, \ldots$
run through strata of dimension
$0, 1, 2, \ldots$
in
$\del S^{(k)}_\circ$
with
$P \subset \del_{in} I,
P, I \subset \del_{in} F,
\ldots$
as above.
In the gluing
\begin{align} \label{eq:k-model}
\widetilde{\bfW}(\Phi_k)
=
[\bfW(\Phi_{k-1})] \#_{\cL_{S^{(k)}}} [T^* \widehat{M}_{S^{(k)}} \times T^* S^{(k)}_\circ]
\end{align}
the cotangent fibers of the images of
$T^* S_P$
get connected through
$J^1 \cL_I$
with the products of
the cotangent fibers of the images of
$T^* (\sigma_P / \langle \sigma^P_I \rangle)$
and
the bases of
$T^* I_\circ$.
The results get connected through
$J^1 \cL_F$
with the product of
the cotangent fibers of the images of
$T^* (\sigma_P / \langle \sigma^P_F \rangle)$
and
the bases of
$T^* F_\circ$.
Inductively,
the results for
$n = k-1$
get connected through
$J^1 \cL_{S^{(k)}}$
with the product of
the cotangent fibers of the image of
$T^* (\sigma_P / \langle \sigma^P_{S^{(k)}} \rangle)$
and
the base of
$T^* S^{(k)}_\circ$.
Hence
\pref{eq:k-symplectomorphism}
respects
the gluing
\pref{eq:k-model}
and
the gluing of
\begin{align*}
T^* T^d \times T^* S_P, \
T^* T^d \times T^* (\sigma_P / \langle \sigma^P_I \rangle) \times T^* I_\circ, \
T^* T^d \times T^* (\sigma_P / \langle \sigma^P_F \rangle) \times T^* F_\circ, \
\ldots
\end{align*}
and
$T^* T^d \times T^* (\sigma_P / \langle \sigma^P_{S^{(k)}} \rangle) \times T^* S^{(k)}_\circ$
induced by the gluing of the zero sections.

By definition
$\bar{\pi}_k$
projects
the former part of
\pref{eq:k-model}
onto the gluing of
\begin{align*}
S_P, \
\sigma_P / \langle \sigma^P_I \rangle \times I_\circ, \
\sigma_P / \langle \sigma^P_F \rangle \times F_\circ, \
\ldots
\end{align*}
and
the latter part onto
$\sigma_P / \langle \sigma^P_{S^{(k)}} \rangle \times S^{(k)}_\circ$
in the gluing
$S$.
Hence
$\bar{\pi}_k$
is compatible with the relevant gluing procedure.
As explained in the proof of
\pref{lem:k-original}(4),
the source
$T^* T^d \times T^* S$
contains
$\pi^{-1}_k(S)$
as the zero section.
Then by
\pref{lem:k-original}(1)
we have
$\bar{\pi}_k |_{\widetilde{\bL}(\Phi_k)}
=
\pi_k$.
\end{proof}

\begin{remark} \label{rmk:k-proper}
If
$\Sigma_S$
is proper for all
$S \subset \Phi$,
then
$\ret_i, i = 0, 1, \ldots, k$
are trivial
and
$\bar{\pi}_k = \tilde{\pi}_k$.
\end{remark}

\begin{remark} \label{rmk:k-tilde}
From
the above proof
and
Definition
\pref{dfn:fanifold}
it follows that
$\tilde{\pi}_k$
projects
the former part of
\pref{eq:k-model}
onto the gluing of
\begin{align*}
M_{P, \bR}, \
M_{I, \bR} \times I_\circ, \
M_{F, \bR} \times F_\circ, \
\ldots
\end{align*}
and the latter part onto
$M_{S^{(k)}, \bR} \times S^{(k)}_\circ$
in the gluing
\begin{align*}
\bigsqcup_P M_{P, \bR}
\#
\bigsqcup_I (M_{I, \bR} \times I_\circ)
\#
\bigsqcup_F (M_{F, \bR} \times F_\circ)
\cdots
\#
(M_{S^{(k)}, \bR} \times S^{(k)}_\circ).
\end{align*}
Hence
$\tilde{\pi}_k$
is also compatible with the relevant gluing procedure.
\end{remark}

\subsection{The associated integrable system}
\begin{definition} \label{dfn:IS}
Let
$(W, \omega)$
be a $2n$-dimensional symplectic manifold,
$B$
an $n$-dimensional manifold
and
$\varpi \colon W \to B$
a smooth surjective map.
A triple
$(W, \varpi, B)$
is an
\emph{integrable system}
if
$\varpi$
satisfies
\begin{align*}
\{ \varpi^{-1} (f), \varpi^{-1} (g) \} = 0, \
f, g \in C^\infty(B)
\end{align*}
where
$\{ \cdot, \cdot \}$
is the Poisson bracket on
$W$.
\end{definition}


\begin{remark}
Here,
we do not require the fibers of
$\varpi$
to be compact.
\end{remark}

\begin{example}
A collection
$(H_1, \ldots, H_n)$
of Hamiltonian functions on
$W$
defines a typical example of integrable systems.
In particular,
on local coordinates
$(q, p) = (q_1, \ldots, q_n, p_1, \ldots, p_n)$
the projection to the base
$(q, p) \mapsto q$
defines an integrable system,
as we have
\begin{align*}
\{ q_i, q_j \}
=
\sum^n_{k = 1} dq_k \wedge dp_k(Z_{q_i}, Z_{q_j})
=
\sum^n_{k = 1} dq_k \wedge dp_k(\del_{p_i}, \del_{p_j})
=
0
\end{align*}
where
$Z_{q_i}$
are
Hamiltonian vector fields for
$q_i$.
Similarly,
the projection to the cotangent fibers
$(q, p) \mapsto p$
defines another integrable system.
\end{example}

\begin{definition}
We call a continuous map from a symplectic manifold
$W$
to
$\Phi$
a
\emph{filtered stratified integrable system}
if its restriction over each filter
$\bar{\Phi}_k \setminus \Phi_{k-1}$
becomes an integrable system in the sense of Definition
\pref{dfn:IS}.
\end{definition}

Consider the map
\begin{align*}
\tilde{\pi}
=
\tilde{\pi}_n
\colon
\widetilde{\bfW}(\Phi)
\to
\bigsqcup_P M_{P, \bR}
\#
\bigsqcup_I (M_{I, \bR} \times I_\circ)
\#
\bigsqcup_F (M_{F, \bR} \times F_\circ)
\cdots
\#
\bigsqcup_{S^{(n)}} (M_{S^{(n)}, \bR} \times S^{(n)}_\circ)
\end{align*}
where
$P, I, F, \ldots, S^{(n)}$
run through strata of dimension
$0, 1, 2, \ldots, n$
of
$\Phi$.
If
\begin{align*}
P \subset \del_{in} I, \
P, I \subset \del_{in} F, \
\ldots
P, I, F, \ldots, S^{(n-1)} \subset \del_{in} S^{(n)},
\end{align*}
then by Definition
\pref{dfn:fanifold}
\begin{align*}
\sigma^P_{S^{(n)}} / \langle \sigma^P_I \rangle, \
\sigma^P_{S^{(n)}} / \langle \sigma^P_F \rangle, \
\ldots
\sigma^P_{S^{(n)}} / \langle \sigma^P_{S^{(n)}} \rangle
\end{align*}
regarded as cones in
$\Sigma_I, \Sigma_F, \ldots, \Sigma_{S^{(n)}}$
do not depend on the choice of
$P$
and
we fix some
$P$
when such
$I, F, \ldots, S^{(n)}$
run. 
By construction
$\tilde{\pi}$
respects the gluing.

\begin{lemma} \label{lem:tilde-IS}
The map
$\tilde{\pi}$
defines a filtered stratified integrable system with noncompact fibers.
\end{lemma}
\begin{proof}
When restricted to each stratum in the filter over
$\bar{\Phi}_k \setminus \Phi_{k-1}$,
clearly
$\widetilde{\bfW}(\Phi)$
becomes a smooth surjective submersion.
Moreover,
$\tilde{\pi}$
is the gluing of products of
the projection from
$T^* \widehat{M}_{S^{(k)}}$
to the cotangent fibers
and
the projection from
$T^* S^{(k)}_\circ$
to the base.
Hence the restriction is an integrable system defined by a collection of Hamiltonian functions.
\end{proof}

Consider the map
\begin{align*}
\bar{\pi} = \bar{\pi}_n \colon \widetilde{\bfW}(\Phi) \to \Phi.
\end{align*}
In general,
$\bar{\pi}$
does not define an integrable system.
For instance,
on $0$-strata
$P_1, \ldots, P_4$
of the fanifold from Example
\pref{eg:closed}
it is not even
$C^1$.
Nevertheless,
from
construction
and
Remark
\pref{rmk:k-proper}
for
$k = n$
it follows 

\begin{lemma}
The map
$\bar{\pi}$
is homotopic to
$\tilde{\pi}$.
If
$\Sigma_S$
is proper for all
$S \subset \Phi$,
then the homotopy becomes trivial
and
$\bar{\pi}$
defines a filtered stratified integrable system with noncompact fibers.
\end{lemma}

\section{SYZ picture}
\subsection{The associated dual stratified spaces}
First,
we recall a fundamental piece of $B$-side SYZ fibrations over stratified spaces dual in a certain sense to fanifolds.
Let
$X_{\Sigma}$
be the $n$-dimensional toric variety associated with a fan
$\Sigma \subset M_\bR$
for a lattice
$M \cong \bZ^n$.
Consider the symplectic moment map
$X_\Sigma \to M^\vee_\bR = \Lie(T_M)^\vee$
for the natural action on
$X_\Sigma$
of the real torus
$T_M \subset M_\bR \otimes_\bZ \bC^*$.
Assume that
$\Sigma$
is the normal fan of a full dimensional lattice polytope
$Q$.
Then the symplectic moment map restricts to a homeomorphism
$(X_\Sigma)_{\geq 0} \to Q$
from nonnegative real points
\cite[Exercise 12.2.8]{CLS}.

Assume further that
$Q$
is very ample.
Then
$Q \cap M^\vee = \{ m_0, \ldots, m_s \}$
defines an embedding
\begin{align*}
X_\Sigma \hookrightarrow \bP^s, \
x \mapsto (\chi^{m_0}(x), \ldots, \chi^{m_s}(x))
\end{align*}
where
$\chi^{m_i} \colon T_{M, \bC} \to \bC^*$
denote the characters.
Recall from
\cite[Section 12.2]{CLS}
that the
\emph{algebraic moment map}
is given by
\begin{align*}
X_\Sigma \to M^\vee_\bR, \
x \mapsto \left( \sum_{m \in Q \cap M^\vee} |\chi^m(x)| \right)^{-1} \sum_{m \in Q \cap M^\vee} |\chi^m(x)|m.
\end{align*}
By
\cite[Theorem 12.2.5]{CLS}
its restriction
$(X_\Sigma)_{\geq 0} \to M^\vee_\bR$
is a homeomorphism to
$Q$.
Note that
the induced homeomorphism depends on the choice of
$Q$.

Consider the map
$X_\Sigma \to (X_\Sigma)_{\geq 0}$
induced by the retraction.
By
\cite[Proposition 12.2.3]{CLS}
the fiber over a point of
$O(\sigma)_{\geq 0}$
for each cone
$\sigma \in \Sigma$
is isomorphic to
$T^{n - \dim \sigma}$.
Here,
$O(\sigma)$
is the orbit corresponding to
$\sigma$
via
\cite[Theorem 3.2.6]{CLS}.
Hence the composition
\begin{align} \label{eq:moment}
\mom_Q
\colon
X_\Sigma
\to
(X_\Sigma)_{\geq 0}
\to
Q
\end{align}
gives a stratified torus fibration.
In fact,
$Q$
gives the dual cell complex to
$\Sigma$.
Since
$\mom_Q$
is compatible with taking subfans,
one can also define it
when
$Q$
is noncompact.

Now,
in order to define the associated dual stratified space,
we assume
$\Phi$
to satisfy the following additional condition.
\begin{itemize}
\item[(vi)]
There is a collection of full dimensional lattice polytopes
$Q_P \subset M^\vee_{P, \bR}$
such that
$\Sigma_P$
are subfans of the normal fans of
$Q_P$
and
for some collection
$\{ l_P \}_{P \in \Phi_0}$
of integers
$l_P Q_P$
are very ample
and
$\mom_{l_P Q_P} (X_{\Sigma_P})$
glue along the inclusions obtained by Definition
\pref{dfn:fanifold}.
\end{itemize}
Here,
we explain more about the condition
(vi).
Given a fanifold
$\Phi$,
we have the disjoint union of toric varieties
$X_{\Sigma_P}$
associated with the fans
$\Sigma_P \subset M_{P, \bR}$
for all $0$-strata.
Consider an exit path
\begin{align*}
P \to S^{(k)} \to \cdots \to S^{(n)}
\end{align*}
to a top dimensional stratum
$S^{(n)}$.
By Definition
\pref{dfn:fanifold}
we have the sequence of quotients
\begin{align*}
T_P \cM \cong M_{P, \bR}
\to
T_{S^{(k)}} \cM |_P \cong M_{S^{(k)}, \bR}
\to
\cdots
\to
T_{S^{(n)}} \cM |_P \cong M_{S^{(n)}, \bR}.
\end{align*}
Taking dual,
we obtain a sequence of inclusions
\begin{align*}
M^\vee_{S^{(n)}, \bR}
\hookrightarrow
\cdots
\hookrightarrow
M^\vee_{S^{(k)}, \bR}
\hookrightarrow
M^\vee_{P, \bR},
\end{align*}
which induces a sequence of inclusions of cones
\begin{align*}
(\sigma_P / \langle \sigma^P_{S^{(n)}} \rangle)^\vee
\hookrightarrow
\cdots
\hookrightarrow
(\sigma_P / \langle \sigma^P_{S^{(k)}} \rangle)^\vee
\hookrightarrow
(\sigma_P)^\vee
\end{align*}
for the cone
$\sigma_P \in \Sigma_P$
corresponding to
$S^{(n)}$.
Note that,
when we fix an inner product on
$M_{P, \bR}$,
\begin{align*}
\Int(\sigma_P)^\vee,
\Int(\sigma_P / \langle \sigma^P_{S^{(k)}} \rangle)^\vee,
\ldots,
\Int(\sigma_P / \langle \sigma^P_{S^{(n)}} \rangle)^\vee
\end{align*}
can be regarded as fibers of
$T_P \cM, T_{S^{(k)}} \cM, \cdots, T_{S^{(n)}} \cM$,
since they are nonempty open star domains of
$M^\vee_{P, \bR}, M^\vee_{S^{(k)}, \bR}, \ldots, M^\vee_{S^{(n)}, \bR}$.

Another exit path
\begin{align*}
P^\prime \to S^{(k)} \to \cdots \to S^{(n)}
\end{align*}
induces a sequence of inclusions of cones
\begin{align*}
(\sigma_{P^\prime} / \langle \sigma^P_{S^{(n)}} \rangle)^\vee
\hookrightarrow
\cdots
\hookrightarrow
(\sigma_{P^\prime} / \langle \sigma^P_{S^{(k)}} \rangle)^\vee
\hookrightarrow
(\sigma_{P^\prime})^\vee
\end{align*}
for the cone
$\sigma_{P^\prime} \in \Sigma_{P^\prime}$
corresponding to
$S^{(n)}$.
Then a choice of compatible identifications
\begin{align*}
(\sigma_P / \langle \sigma^P_{S^{(n)}} \rangle)^\vee
\cong
(\sigma_{P^\prime} / \langle \sigma^{P^\prime}_{S^{(n)}} \rangle)^\vee, \
\ldots,
(\sigma_P / \langle \sigma^P_{S^{(k)}} \rangle)^\vee
\cong
(\sigma_{P^\prime} / \langle \sigma^{P^\prime}_{S^{(k)}} \rangle)^\vee
\end{align*}
for all strata
$S^{(k)}$
with
$P, P^\prime \subset \bar{S}^{(k)}$
and
$S^{(k)} \subset \bar{S}^{(n)}$
give a gluing datum for
$(\sigma_P)^\vee$
and
$(\sigma_{P^\prime})^\vee$.
Identifying
$(\sigma_P)^\vee$
with the image
$\mom_{l_P Q_P}(X_{\Sigma_P \cap \sigma_P})$
for each
$P \in \Phi_0$,
the condition
(vi)
requires the existence of such gluing data compatible with all exit paths. 

\begin{definition} \label{dfn:dual}
Let
$\Phi$
be a fanifold satisfying the additional condition
(vi).
We define its
\emph{associated dual stratified space}
$\Psi$
as the  gluing of
$\mom_{l_P Q_P}(X_{\Sigma_P})$
with the canonical stratification in a sufficiently large ambient space
$\cN$.
For a $k$-stratum
$S^{(k)}$
of
$\Phi$,
its
\emph{dual stratum}
$S^{(k), \perp}$
of
$\Psi$
is the stratum defined by the image
$\mom_{l_P Q_P}(O(\sigma^P_{S^{(k)}}))$.
\end{definition}

The dual stratified space
$\Psi$
is well defined up to canonical stratified isomorphism.
Indeed,
it depends only on
$\mom_{l_P Q_P}(X_{\Sigma_P})$
for a single $0$-stratum
$P \in \Phi_0$,
which in turn amounts to fixing a dual cell complex of
$\Sigma_P$.
Then the condition
(vi)
automatically fixes that of the other $0$-strata.
On the other hand,
the image
$\mom_{l^\prime_P Q^\prime_P}(X_{\Sigma_P})$
for another
full dimensional lattice polytope
$Q^\prime_P$
and
integer
$l^\prime_P$
might differ from
$\mom_{l_P Q_P}(X_{\Sigma_P})$
only by
translation
and
dilation,
both of
which respect the canonical stratifications.

\begin{example}
Let
$\Phi$
be the fanifold from Example
\pref{eg:closed}.
Adding rays to
$P_1, P_2, P_3, P_4$
parallel to vectors
$(-1,1), (-1,-1), (1,-1), (1,1)$,
we obtain the normal fans
$\Sigma_{Q_{P_i}}$
of full dimensional lattice polytopes
$Q_{P_i}$.
Since we have
$\dim M_{P_i, \bR} = 2$,
by
\cite[Corollary 2.2.19]{CLS}
any full dimensional lattice polytope is very ample.
Then
$\mom_{Q_{P_i}}(X_{\Sigma_{P_i}})$
glue to yield a stratified space
$\Psi \subset M^\vee_\bR$.
Its
$0$-stratum
$F^\perp$
is a point
$(1/2, 1/2)$,
$1$-strata
$I^\perp_{12}, I^\perp_{23}, I^\perp_{34}, I^\perp_{14}$
are defined by rays from
$F^\perp$
parallel to vectors
$(-1, 0), (0, -1), (1, 0), (0, 1)$,
and
$2$-strata
$P^\perp_1, P^\perp_2, P^\perp_3, P^\perp_4$
are defined by quadrants with the origin placed at
$F^\perp$
bounded by the pairs
$(I^\perp_{12}, I^\perp_{14}),
(I^\perp_{12}, I^\perp_{23}),
(I^\perp_{23}, I^\perp_{34}),
(I^\perp_{14}, I^\perp_{34})$.
\end{example}

\begin{example}[{\cite[Example 4.24]{GS2}}]
Let
$\Sigma^\circ$
be the fan obtained from the standard fan
$\Sigma_{\bA^3} \subset \bR^3$
for the toric variety
$\bA^3$
by deleting the rays.
It has a fanifold structure inherited by the canonical one on
$\Sigma^\circ \cap S^2$,
which is a $2$-simplex without its vertices.
The absence of vertices immediately imply the failure of the condition
(vi). 
\end{example}

\begin{lemma} \label{lem:dual-filtration}
Let
$\Phi$
be a fanifold of dimension
$n$
satisfying the additional condition
(vi).
Then its associated dual stratified space
$\Psi$
admits a filtration
\begin{align} \label{eq:filtration2}
\Psi_0 \subset \Psi_1 \subset \cdots \subset \Psi_n = \Psi
\end{align}
where
$\Psi_k$
is a stratified subspace defined as the complement in
$\Psi$
of $k$-skeleta
$\Sk_{n-k-1}(\Psi)$,
the closure of the subset of $n-k-1$-strata with
$\Sk_{-1}(\Psi) = \emptyset$.
\end{lemma}
\begin{proof}
It follows immediately from the construction of
$\Psi$.
\end{proof}

In
\cite[Section 3]{GS1},
to a fanifold
$\Phi$
Gammage--Shende associated the colimit
\begin{align*}
\bfT(\Phi) = \varinjlim_S X_{\Sigma_S}
\end{align*}
along closed embeddings induced by quotient maps between fans,
where
$S$
runs through all strata of
$\Phi$.
By
\cite[Proposition 3.10]{GS1}
the colimit
$\bfT(\Phi)$
always exists as an algebraic space.
We will recall later that
$\bfT(\Phi)$
is a mirror partner of
$\widetilde{\bfW}(\Phi)$.
When
$\Phi$
admits the associated dual stratified space,
$B$-side SYZ fibration over
$\Psi$
for the pair
$(\widetilde{\bfW}(\Phi), \bfT(\Phi))$
should be the following. 

\begin{definition}
Let
$\Phi$
be a fanifold satisfying the additional condition
(vi).
Fix its associated dual stratified space
$\Psi$.
We define
\begin{align*}
\mom_\Phi \colon \bfT(\Phi) \to \Psi
\end{align*}
as the gluing of
$\mom_{l_P Q_P}$
where
$P$
runs through all $0$-strata of
$\Phi$.
\end{definition}

Note that
gluing of
$\mom_{l_P Q_P}(X_{\Sigma_P})$
for all $0$-strata
$P$
descends to that of
$\mom_{l^P_{S^{(k)}} Q^P_{S^{(k)}}}(X_{\Sigma_{S^{(k)}}})$
for all $k$-strata
$S^{(k)}$
where
$Q^P_{S^{(k)}}, l^P_{S^{(k)}}$
denote the induced
polytopes
and
integers.
Since both closed embeddings
$X_{\Sigma_{S^{(k)}}}
\hookrightarrow
X_{\Sigma_P}$
and
$(\sigma_P / \langle \sigma^P_{S^{(k)}} \rangle)^\vee
\hookrightarrow
(\sigma_P)^\vee$
correspond to the quotient
$\Sigma_P
\twoheadrightarrow
\Sigma_{S^{(k)}}
=
\Sigma_P / \sigma^P_{S^{(k)}}$,
by definition of
$\Psi$
we have the commutative diagram
\begin{align*}
\begin{gathered}
\xymatrix{
X_{\Sigma_P \cap \sigma_P} \ar[d]_{\mom_{l_P Q_P}} & X_{\Sigma_{{S^{(k)}}} \cap (\sigma_P / \langle \sigma^P_{S^{(k)}} \rangle)} \ar@{_{(}->}[l] \ar[d]_{\mom_{l^P_{S^{(k)}} Q^{P}_{S^{(k)}}}} \ar@{=}[r] & X_{\Sigma_{S^{(k)}} \cap (\sigma_{P^\prime} / \langle \sigma^{P^\prime}_{S^{(k)}} \rangle)} \ar@{^{(}->}[r] \ar[d]^{\mom_{l^{P^\prime}_{S^{(k)}} Q^{P^\prime}_{S^{(k)}}}} & X_{\Sigma_{P^\prime} \cap \sigma_{P^\prime}} \ar[d]^{\mom_{l_{P^\prime} Q_{P^\prime}}} \\
(\sigma_P)^\vee  & (\sigma_P / \langle \sigma^P_{S^{(k)}} \rangle)^\vee \ar@{_{(}->}[l] \ar@{=}[r] & (\sigma_{P^\prime} / \langle \sigma^{P^\prime}_{S^{(k)}} \rangle)^\vee \ar@{^{(}->}[r] & (\sigma_{P^\prime})^\vee.
}
\end{gathered}
\end{align*}
Hence
$\mom_{l_P Q_P}$
glue to yield
$\mom_\Phi$.

\subsection{Proof of \pref{thm:dual}} 
Suppose that
$\Phi$
satisfies the additional condition
(vi).
Fix its associated dual stratified space
$\Psi$.
By
\pref{lem:dual-filtration}
it admits a filtration
\pref{eq:filtration2}.
Fix further a dual cell complex of
$\Psi$
in
$\cN$
such that each of its $0$-stratum
$P^\dagger$
is contained in the corresponding top dimensional stratum
$P^\perp$
of
$\Psi$.
Combinatorially,
it can be obtained by gluing the subfans
$\Sigma_P$
of the normal fans of
$Q_P$.
Now,
we construct a fibration
$\underline{\pi}
\colon
\widetilde{\bfW}(\Phi)
\to
\Psi$
as an integrable system with noncompact fibers,
which should be SYZ dual to
$\mom_\Phi$.
We proceed by induction on
$k$.
In the sequel,
for simplicity we will assume that
all 
$\Sigma_P$
are proper.

Choose an orthonormal basis of
$M_{P, \bR}$
for each
$P \in \Phi_0$.
Let
$D_{P, \epsilon}$
be the disk of radius
$\epsilon$
around the origin
$0 \in \Sigma_P \subset M_{P, \bR}$
with respect to the basis.
Then dilation gives a diffeomorphism
$D_{P, \epsilon} \to M_{P, \bR}$.
Let
$D_{P^\dagger, \epsilon}$
be the disk of radius
$\epsilon$
around
$P^\dagger \in P^\perp \subset M^\vee_{P, \bR}$
with respect to the dual basis.
Taking
$\epsilon$
sufficiently small,
we may assume that
each
$D_{P^\dagger, \epsilon}$
is contained in
$P^\perp$,
as there are only finitely many strata.
Then projection from
$P^\dagger$
gives a homeomorphism
$\bar{D}_{P^\dagger, \epsilon} \to \bar{P}^\perp$
of the closures,
which induces a diffeomorphism
$D_{P^\dagger, \epsilon} \to P^\perp$
and
a stratification of
$\del \bar{D}_{P^\dagger, \epsilon}$.

\begin{definition}
We define
\begin{align*}
\underline{\pi}_0
\colon
\widetilde{\bfW}(\Phi_0)
\to
\Psi_0
\end{align*}
as the composition of
$\tilde{\pi}_0$
with the disjoint union of diffeomorphisms
\begin{align*}
M_{P, \bR}
\to
D_{P, \epsilon}
\to
D_{P^\dagger, \epsilon}
\to
P^\perp,
\end{align*}
where each
$D_{P, \epsilon} \to D_{P^\dagger, \epsilon}$
is the stratified diffeomorphism induced by
$\Sigma_P$.
\end{definition}

\begin{remark} \label{rmk:moment}
The map
$\underline{\pi}_0$
can be regarded as the disjoint union of moment maps
\begin{align*}
\Log_P \colon T^* \widehat{M}_P \cong (\bC^*)^n \to \bR^n, \
(\theta, \xi) \mapsto \xi
\end{align*}
for the lifts to
$T^* \widehat{M}_P$
of the self $\widehat{M}_P$-actions.
\end{remark}

For each $k$-stratum
$S^{(k)}$
adjacent to
$P$,
let
$D_{S^{(k)}, P, \epsilon} \subset \del \bar{D}_{P, \epsilon}$
be the inverse image of
$S^{(k), \perp}$
under the homeomorphism
$\bar{D}_{P, \epsilon}
\to
\bar{D}_{P^\dagger, \epsilon}
\to
\bar{P}^\perp$.
It has the induced stratification by
$\Sigma_P$.
There is a canonical diffeomorphism
$M_{S^{(k)}, \bR} \to D_{S^{(k)}, P, \epsilon}$
which sends the stratification of
$\Sigma_{S^{(k)}}$
to that of
$D_{S^{(k)}, P, \epsilon}$.
Note that
we will have this diffeomorphism even without assuming
$\Sigma_P$
to be proper.
We denote by
$\underline{\pi}_{0, S^{(k)}}$
the composition of
$\tilde{\pi}_{0, S^{(k)}}$
with the diffeomorphism
\begin{align*}
M_{S^{(k)}, \bR}
\to
D_{S^{(k)}, P, \epsilon}
\to
S^{(k), \perp}.
\end{align*}
This is well defined.
Namely,
for all exit paths
$P \to S^{(k)} \leftarrow P^\prime$
the diagrams
\begin{align*}
\begin{gathered}
\xymatrix{
M_{P, \bR} \ar[r] & M_{S^{(k)}, \bR} \ar[r] \ar@{=}[d] & D_{S^{(k)}, P, \epsilon} \ar[r] & S^{(k), \perp} \ar@{=}[d] \\
M_{P^\prime, \bR} \ar[r] & M_{S^{(k)}, \bR} \ar[r] &  D_{S^{(k)}, P^\prime, \epsilon} \ar[r] & S^{(k), \perp}.
}
\end{gathered}
\end{align*}
commute.
Indeed,
our definition is independent of the choices of orthonormal bases of
$M_{P, \bR}$,
as we use only
dilations,
projections from
$P^\dagger$
to
$\del \bar{P}^\perp$
and
compatible quotient sequences
\begin{align*}
\begin{gathered}
M_{P, \bR} \to M_{I, \bR} \to M_{F, \bR} \to \cdots \\
\Sigma_P \to \Sigma_I \to \Sigma_F \to \cdots
\end{gathered}
\end{align*}
for all
$P \in \Phi_0$.
Recall that
the fixed dual cell complex of
$\Psi$
in
$\cN$
induces
stratifications of
$D_{P^\dagger, \epsilon}$
by
$\Sigma_P$
and
the maps
$D_{P, \epsilon} \to D_{P^\dagger, \epsilon}$
which respect the stratifications.
Note that
we will also have this diffeomorphism even without assuming
$\Sigma_P$
to be proper,
as 
$\Sigma_P$
are simplicial.

Suppose that
we construct
$\underline{\pi}_{k-1}$
for the subfanifold
$\Phi_{k-1}$.

\begin{definition}
We define
\begin{align*}
\underline{\pi}_k
\colon
\widetilde{\bfW}(\Phi_k)
\to
\Psi_k
\end{align*}
as follows.
First,
consider the map canonically induced by
$\underline{\pi}_{k-1}, \underline{\pi}_{0, S^{(k)}}$
and
the projections
$T^* S^{(k)}_\circ \to S^{(k)}_\circ$
to the base.
Next,
contract the images
$S^{(k)}_\circ$
to obtain
$\underline{\pi}_k$.
\end{definition}

\begin{remark} \label{rmk:k-underline}
Since we contract
$T^* S^{(k)}_\circ$
to single points
when defining
$\underline{\pi}_k$,
by Remark
\pref{rmk:k-tilde}
one may regard
$\underline{\pi}_k$
as the gluing of
$\underline{\pi}_{k-1}$
with
$\underline{\pi}_{0, S^{(k)}}$.
As the contraction of the images
$S^{(k)}_\circ$
corresponds to shrinking the Lagrangians
$\widehat{M}_{S^{(k)}} \times S^{(k)}_\circ$
in the gluing
$\widetilde{\bL}(\Phi_k)$,
the image of
$\underline{\pi}_k$
is homeomorphic to that of
$\tilde{\pi}_k$.
\end{remark}

We denote
$\underline{\pi}_n$
by
$\underline{\pi}$.
From Remark
\pref{rmk:k-underline}
for
$k = n$
one sees that
$\underline{\pi}$
is the gluing of the disjoint unions of the projections to the cotangent fibers from
$T^* \widehat{M}_P$
along the canonical inclusions induced by
\begin{align*}
\widehat{M}_{S^{(n)}}
\hookrightarrow
\ldots
\hookrightarrow
\widehat{M}_F
\hookrightarrow
\widehat{M}_I
\hookrightarrow
\widehat{M}_P,
\end{align*}
which are obtained by Definition
\pref{dfn:fanifold}.
Here,
$P, I, F, \ldots, S^{(n)}$
run through strata of dimension
$0, 1, 2, \ldots, n$
of
$\Phi$
with
$P \subset \del_{in} I,
P, I \subset \del_{in} F,
\ldots,
P, I, F, \ldots, S^{(n-1)} \subset \del_{in} S^{(n)}$.

\begin{lemma}
The map
$\underline{\pi}$
defines a stratified fibration,
whose fiber over a point of any $k$-stratum
$S^{(k), \perp}$
is given by
$T^k \times T^* S^{(k)}$. 
\end{lemma}
\begin{proof}
It immediately follows from
definition
and
Remark
\pref{rmk:k-underline}
\end{proof}

\subsection{Very affine hypersurfaces}
First,
we recall HMS for very affine hypersurfaces established by Gammage--Shende.
Let
$\Sigma \subset M_\bR \cong \bR^{n+1}$
be a smooth quasiprojective stacky fan
\cite[Definition 2.4]{GS15}
whose primitive ray generators span a convex lattice polytope
$\Delta^\vee$
containing the origin.
Then
$\Sigma$
defines
a smooth Deligne--Mumford stack
$\bfT_\Sigma$
\cite[Definition 2.5]{GS15}
with toric boundary divisor
$\del \bfT_\Sigma$
and
an adapted star-shaped triangulation
$\cT$
of
$\Delta^\vee$
\cite[Definition 3.3.1]{GS2}.
Recall that a triangulation
$\cT$
is
\emph{adapted}
if there is a convex piecewise function
$\mu \colon \Delta^\vee \to \bR$
whose corner locus is
$\cT$.
We denote by
$\bT_\bC$
the complex torus
$M_\bR / M \otimes_\bR \bC^*$
acting on
$\bfT_\Sigma$,
where
$\bT$
is a real $(n+1)$-dimensional torus with
character lattice
$M^\vee$
and
cocharacter lattice
$M$.

Consider a Laurent polynomial
\begin{align*} 
W_t \colon \bT^\vee_\bC \to \bC, \
z
\mapsto
\sum_{\alpha \in \VVert(\cT)} c_\alpha t^{-\mu(\alpha)}z^\alpha
\end{align*}
in coordinates
$z = (z_1, \ldots, z_{n+1})$
on
$\bT^\vee_\bC$,
where
$c_\alpha \in \bC^*$
are arbitrary constants
and
$t \gg 0$
is a tropicalization parameter.
For sufficiently general
$t$,
the hypersurface
$H_t = W^{-1}_t(0)$
is
smooth
and
called
\emph{very affine}.
Restricting the standard Liouville structure on
$\bT^\vee_\bC \cong T \bT^\vee \cong T^* \bT^\vee$
for a fixed inner product,
we may regard
$H_t$
as a Liouville submanifold.

\begin{theorem}[{\cite[Theorem 1.0.1]{GS2}}]
There is an equivalence
\begin{align*}
\Fuk(H_t) \simeq \Coh(\del \bfT_\Sigma)
\end{align*}
between
the wrapped Fukaya category on
$H_t$
and
the dg category of coherent sheaves on
$\del \bfT_\Sigma$.
\end{theorem}

Next,
following
\cite[Section 3]{AAK},
we review the SYZ fibrations associated with mirror symmetry for very affine hypersurfaces.
For simplicity,
we temporarily assume that
$\Sigma$
is an ordinary simplicial fan.
Consider the moment map
\begin{align*} 
\Log \colon \bT^\vee_\bC \to \bR^{n+1}, \
z \mapsto (\log|z_1|, \ldots, \log |z_{n+1}|)
\end{align*}
for the lift to
$T^* \bT^\vee$
of the self $\bT^\vee$-action.
The image
$\Pi_t = \Log(H_t)$
is called the
\emph{amoeba}
of
$H_t$.

\begin{definition}
The
\emph{tropical hypersurface}
$\Pi_\Sigma$
associated with
$H_t$
is the hypersurface defined by the
\emph{tropical polynomial}
\begin{align*}
\varphi \colon M^\vee_\bR \to \bR, \
\varphi(m)
=
\max\{\langle m, n \rangle - \mu(n) \ | \ n \in \Delta^\vee \}.
\end{align*}
\end{definition}

It is known that
$\Pi_\Sigma$
is a deformation retract of
$\Pi_t$.
According to
\cite[Corollary 6.4]{Mik},
when
$t \to \infty$
the rescaled amoeba
$\Pi_t / \log t$
converges to
$\Pi_\Sigma$.
Combinatorially,
$\Pi_\Sigma$
is the dual cell complex of
$\cT$.
In particular,
the set of connected components of
$\bR^{n+1} \setminus \Pi_\Sigma$
bijectively corresponds to the set
$\VVert(\cT)$
of vertices of
$\cT$
according to which
$\alpha \in \VVert(\cT)$
achieves the maximum of
$\langle m, \alpha \rangle - \mu(\alpha)$
for
$m \in \bR^{n+1} \setminus \Pi_\Sigma$.
Note that
$\bR^{n+1} \setminus \Pi_t$
for
$t \gg 0$
has the same combinatrics as
$\bR^{n+1} \setminus \Pi_\Sigma$.
In the sequel,
we will fix a general
$t \gg 0$
and
drop
$t$
from the notation.

Let
$\ret \colon \Pi \to \Pi_\Sigma$
be the continuous map induced by the retraction.
Then the composition
\begin{align} \label{eq:A-fibration}
H
\xrightarrow{\Log |_H}
\Pi
\xrightarrow{\ret}
\Pi_\Sigma
\end{align}
gives the $A$-side SYZ fibration.
Recall that
$H$
admits a pants decomposition
\cite[Theorem 1']{Mik}.
By
\cite[Proposition 4.6]{Mik}
the $k$-th intersection of $i_1, \cdots, i_k$-th legs
\cite[Definition 5.2.2]{GS2}
of an $n$-dimensional tailored pants
$\tilde{P}_n$
is isomorphic to a product
\begin{align*}
\bC^*_{z_{i_1}} \times \cdots \times \bC^*_{z_{i_k}} \times \tilde{P}_{n-k}.
\end{align*}
Under
\pref{eq:A-fibration}
a $k$-th intersection of legs maps to a subset of $k$-stratum away from lower dimensional strata.
Here,
we equip
$\Pi_\Sigma$
the canonical stratification.
In particular,
the fiber over a general point of a $k$-stratum contains
$T^k$.
On the other hand,
the $B$-side SYZ fibration is induced by the map from the total space of the anticanonical sheaf on
$\bfT_\Sigma = X_\Sigma$
defined as
\pref{eq:moment}.

Now,
we return to the case
where
$\Sigma$
is a smooth quasiprojective stacky fan.
Since
$\cT$
is star-shaped,
each $(n+1)$-simplex is a polytope
\begin{align*}
\Delta^\vee_\sigma
=
\Conv(0, \alpha^1, \ldots, \alpha^{n+1})
\subset
M_\bR
\end{align*}
spanned by
the origin
and
primitive ray generators
$\alpha^1, \ldots, \alpha^{n+1}$
of some maximal dimensional cone
$\sigma$.
Consider the map
\begin{align*}
\Delta^\vee_{n+1} \to \Delta^\vee_\sigma, \
e_i \mapsto \alpha^i  
\end{align*}
from the standard $(n+1)$-simplex with the standard basis
$e_1, \ldots, e_{n+1}$,
whose dual induces a map
$f_{\Delta^\vee_\sigma} \colon \bT^\vee_\bC \to \bT^\vee_\bC$.

\begin{definition}[{\cite[Definition 5.1.4]{GS2}}] 
The
\emph{$\Delta^\vee_\sigma$-pants}
is the inverse image
$\tilde{P}_{\Delta^\vee_\sigma} = f^{-1}_{\Delta^\vee_\sigma}(\tilde{P}_n)$
of the tailored pants.
We write
$\tilde{A}_{\Delta^\vee_\sigma}$
for its amoeba
$\Log(\tilde{P}_{\Delta^\vee_\sigma})$.
For
$b = (b_1, \ldots, b_{n+1}), b_i \gg 0$
we denote by
$\tilde{P}^b_{\Delta^\vee_\sigma}$
and
$\tilde{A}^b_{\Delta^\vee_\sigma}$
respectively
the
\emph{translated tailored $\Delta^\vee_\sigma$-pants}
obtained by scaling the coefficients of the defining polynomial by
$e^{- b_i}$
and
its amoeba obtained by translation to the first orthant.
\end{definition}

\begin{lemma}[{\cite[Lemma 5.3.6]{GS2}}] \label{lem:localcore} 
Let
$\del^0 \tilde{A}^b_{\Delta^\vee_\sigma}$
be the component of
$\del \tilde{A}^b_{\Delta^\vee_\sigma}$
which bounds the region of
$\bR^{n+1}$
containing the inverse image of the all-negative orthant.
Restrict the canonical Liouille structure on
$T^* \bT^\vee$
to
$\tilde{P}^b_{\Delta^\vee_\sigma}$.
Let
$\Sigma_{\sigma} \subset \Sigma$
be the stacky subfan whose primitive ray generators are
$\alpha^1, \ldots, \alpha^{n+1}$.
Then we have  
\begin{align*}
\Core(\tilde{P}^b_{\Delta^\vee_\sigma})
=
\Log^{-1}(\del^0 \tilde{A}^b_{\Delta^\vee_\sigma})
\cap
\bL(-\Sigma_\sigma).
\end{align*}
\end{lemma}

From
\cite[Theorem 6.2.4]{GS2}
it follows that
$\Core(H)$
is the gluing of
$\Core(\tilde{P}^b_{\Delta^\vee_\sigma})$
for all
$\sigma \in \Sigma_{\max}$.
Hence in this case the fiber of
\pref{eq:A-fibration}
over a point of $k$-stratum away from lower dimensional strata becomes isomorphic to a finite cover of a product
\begin{align*}
\bC^*_{z_{i_1}} \times \cdots \times \bC^*_{z_{i_k}} \times \tilde{P}^b_{\Delta^\vee_{\sigma / \langle i_1, \ldots, i_k \rangle}}.
\end{align*}
Here,
$\tilde{P}^b_{\Delta^\vee_{\sigma / \langle i_1, \ldots, i_k \rangle}}
=
f^{-1}_{\Delta^\vee_{\sigma / \langle i_1, \ldots, i_k \rangle}}(\tilde{P}_{n-k})$
is the finite cover of
$\tilde{P}_{n-k}$
determined by
\begin{align*}
\Delta^\vee_{\sigma / \langle i_1, \ldots, i_k \rangle}
=
\Conv(0, \alpha^1, \ldots, \hat{\alpha}^{i_1}, \ldots, \hat{\alpha}^{i_k}, \ldots, \alpha^{n+1}).
\end{align*}
On the other hand,
the $B$-side SYZ fibration becomes the composition of the structure morphism
$\bfT_\Sigma \to X_\Sigma$
to the coarse moduli space with the map defined as
\pref{eq:moment}.
The fiber over any point of a $k$-stratum is the real part of the corresponding subgroup of the Deligne--Mumford torus
\cite[Definition 2.4, Proposition 2.6]{FMN}
acting on
$\bfT_\Sigma$.

\subsection{Proof of \pref{thm:restriction}}
For a smooth quasiprojective stacky fan
$\Sigma \subset M_\bR$,
consider the fanifold 
$\Phi = \Sigma \cap S^{n+1}$
from
\cite[Example 4.22]{GS1}
generalized as in
\cite[Section 6]{GS1}.
In particular,
to each $0$-stratum
$P$
we associate the stacky fan
$\Sigma_P = \Sigma / \rho_P$
where
$\rho_P \in \Sigma$
is the ray passing through
$P$.
Hence
$\bfT(\Phi)$
coincide with the toric boundary divisor
$\del \bfT_\Sigma$
of
$\bfT_\Sigma$. 
Fix a stratified homeomorphism
$h_\Phi \colon \del \Delta^\vee \to \Phi$.
Then strata of
$\del \Delta^\vee$
inherit the labels from
$\Phi$.

\begin{definition}
Let
$\Pi^0_\Sigma$
be the connected component of
$\bR^{n+1} \setminus \Pi_\Sigma$
corresponding to the origin of
$\Delta^\vee$.
We define
$\Psi$
as the image of a fixed stratified homeomorphism
$h_\Psi \colon \del \Pi^0_\Sigma \hookrightarrow S^{n+1} \subset M^\vee_\bR$,
which makes the inherited labels on strata of
$\Psi$
compatible with that on strata of
$\Phi$
transported to
$M^\vee_\bR$
via a fixed inner product.
\end{definition} 

Note that
$\del \Pi^0_\Sigma$
is the image of the composition
\begin{align*}
\del \bfT_\Sigma \hookrightarrow \bfT_\Sigma \to X_\Sigma \to (X_\Sigma)_{\geq 0} \to M^\vee_\bR.
\end{align*}
Hence
$\Phi$
satisfies
(vi)
and
admits the dual stratified space
$\Psi$.
In particular,
the $B$-side fibration
$\bfT(\Phi) \to \Psi$
is
the gluing of the compositions
$\bfT_{\Sigma_P} \to X_{\Sigma_P} \to l_P Q_P$
and
compatible with that
$\del \bfT_\Sigma \to \del \Pi^0_\Sigma$
for very affine hypersurface.
Namely,
we have the commutative diagram
\begin{align*}
\begin{gathered}
\xymatrix{
\bfT(\Phi) \ar@{=}[r] \ar[d] & \del \bfT_\Sigma \ar[d] \\
\Psi \ar[r]^{h^{-1}_\Psi} & \del \Pi^0_\Sigma.
}
\end{gathered}
\end{align*}

\begin{lemma}
There is a stratified homeomorphism of
the image of
$\Core(H)$
under the composition
$\ret \circ \Log$
and
$\Phi$
transported to
$M^\vee_\bR$
via the fixed inner product.
\end{lemma}
\begin{proof}
As explained above,
$\Core(H)$
is the gluing of
$\Core(\tilde{P}^b_{\Delta^\vee_\sigma})$
for all
$\sigma \in \Sigma_{\max}$.
By
\pref{lem:localcore}
each piece is the Legendrian boundary
$\del_\infty \bL(-\Sigma_\sigma)$
of the FLTZ Lagrangian
$\bL(-\Sigma_\sigma)$
associated with the stacky subfan
$\Sigma_\sigma \subset \Sigma$.
Unwinding the proof of
\cite[Theorem 5.13]{Nad},
one sees that
$\ret \circ \Log$
projects each trajectory of $k$-dimensional torus in
$\del_\infty \bL(-\Sigma_\sigma)$
along the Liouville flow to a connected subset of the corresponding cone in
$\Sigma_\sigma(n+2-k)$
transported to
$M^\vee_\bR$.
Such trajectories of $k$-dimensional tori are connected with that of lower dimensional tori in
$\del_\infty \bL(-\Sigma_\sigma)$.
Hence the composition
\begin{align*}
h_\Psi \circ \ret \circ \Log |_{\Core(H)}
\colon
\Core(H) \to \Psi
\end{align*}
projects the trajectory to a connected subset of the corresponding $(n+1-k)$-stratum of
$\Phi$
transported to
$M^\vee_\bR$.
Note that
tori in different
$\del_\infty \bL(-\Sigma_\sigma)$
mapping to the same point get identified in
$\Core(H)$.
The images of the trajectories of $k$-dimensional tori are connected with that of lower dimensional tori
and
the images of the Liouville flow become parallel to the corresponding strata of
$\Phi$.
Thus one can find a homeomorphism of the image of each stratum of
$\Core(H)$
under
$h_\Psi \circ \ret \circ \Log$
to the corresponding stratum of
$\Phi$
transported to
$M^\vee_\bR$.
By construction such homomorphisms glue to yield the desired map.
\end{proof}

\begin{corollary}
There is a diffeomorphism of
$\widetilde{\bL}(\Phi)$
and
$\Core(H)$
over
$\Phi$
transported to
$M^\vee_\bR$.
\end{corollary}
\begin{proof}
For any $k$-stratum
$S^{(k)}_P$
of
$\Phi$
adjacent to a $0$-stratum
$P$,
its inverse image under
$h_\Psi \circ \ret \circ \Log |_{\Core(H)}$
is diffeomorphic to the zero section of
$T^* \widehat{M}_{S^{(k)}_P} \times T^* S^{(k)}_P$.
By construction
\begin{align*}
(h_\Psi \circ \ret \circ \Log |_{\Core(H)})^{-1}(S^{(k)}_{P_\circ}), \
\widehat{M}_{S^{(k)}_P} \times S^{(k)}_{P, \circ}
\end{align*}
respectively glue along the boundaries to yield
$\Core(H), \widetilde{\bL}(\Phi)$.
\end{proof}

\begin{corollary}
There is a symplectomorphism of pairs
\begin{align}
(\widetilde{\bfW}(\Phi), \widetilde{\bL}(\Phi))
\hookrightarrow
(H, \Core(H))
\end{align}
over
$\Psi$
with respect to
$\underline{\pi}$
and
the modification of
$h_\Psi \circ \ret \circ \Log$
by further retraction
$\Pi_\Sigma = \ret \circ \Log(H)$
to
$\del \Pi^0_\Sigma$
before composing
$h_\Psi$.
\end{corollary}
\begin{proof}
Cotangent scaling makes
$(T^* \widehat{M}_{S^{(k)}_P} \times T^* S^{(k)}_P,  \widehat{M}_{S^{(k)}_P} \times S^{(k)}_P)$
symplectomorphic to
\begin{align*}
(\Nbd((h_\Psi \circ \ret \circ \Log |_{\Core(H)})^{-1}(S^{(k)}_P)), (h_\Psi \circ \ret \circ \Log |_{\Core(H)})^{-1}(S^{(k)}_P))
\end{align*}
for any $k$-stratum
$S^{(k)}_P$
of
$\Phi$
adjacent to a $0$-stratum
$P$.
By construction they glue to yield a symplectomorphism of pairs
\begin{align*}
(\widetilde{\bfW}(\Phi), \widetilde{\bL}(\Phi)) \hookrightarrow (H, \Core(H))
\end{align*}
over the pair
$(\tilde{\pi}(\widetilde{\bfW}(\Phi)), \Phi)$.
The same modification as in the proof of
\pref{thm:dual}
yields compatible fibrations over
$\Psi$.
\end{proof}


\end{document}